\documentclass[12 pt]{amsart}
\usepackage{amssymb, amsmath, amsfonts, amsthm, graphics,mathrsfs, mathtools}
\usepackage[usenames, dvipsnames]{xcolor}
\usepackage[hmargin=1 in, vmargin = 1 in]{geometry}
\usepackage{tikz}
\usetikzlibrary{matrix, calc, arrows, positioning, backgrounds} 
\usepackage{hyperref}

\usepackage{ytableau}

\definecolor{darkblue}{rgb}{0.0,0,0.7}

\newcommand{\newword}[1]{\textcolor{darkblue}{\textbf{\emph{#1}}}}

\newcommand{\GL}{\mathrm{GL}}
\newcommand{\Flags}{\mathsf{Flags}}

\newcommand{\id}{\mathrm{id}}

\newcommand{\CC}{\mathbb{C}}

\newcommand{\maj}{\mathrm{maj}}

\newcommand{\plus}{\textcolor{purple}{+}}
\newcommand{\minus}{\textcolor{purple}{-}}
\newcommand{\sign}{\textcolor{purple}{\pm}}
\newcommand{\revsign}{\textcolor{purple}{\mp}}

\newcommand{\Des}{\mathrm{Des}}

\newtheorem{Theorem}{Theorem}[section]

\newtheorem{corollary}[Theorem]{Corollary}

\newtheorem{proposition}[Theorem]{Proposition}

\newtheorem{theorem}[Theorem]{Theorem}

\newtheorem{lemma}[Theorem]{Lemma}

\theoremstyle{definition}

\newtheorem{remark}[Theorem]{Remark}

\numberwithin{equation}{section}

\theoremstyle{remark}
\newtheorem{examplex}[Theorem]{Example}
\newenvironment{example}
  {\pushQED{\qed}\examplex}
  {\popQED\endexamplex}

\title{An inverse Grassmannian Littlewood--Richardson rule and extensions}
\author{Oliver Pechenik}
\address[OP]{Department of Combinatorics \& Optimization, University of Waterloo, Waterloo ON N2L3G1, Canada}
\email{oliver.pechenik@uwaterloo.ca}
\author{Anna Weigandt}
\address[AW]{School of Mathematics, University of Minnesota, Minneapolis MN 55455, USA}
\email{weigandt@umn.edu}

\date{\today}

\makeatletter
\@namedef{subjclassname@2020}{%
  \textup{2020} Mathematics Subject Classification}
\makeatother


\begin{document}

\begin{abstract}
Chow rings of flag varieties have bases of Schubert cycles $\sigma_u$, indexed by permutations. A major problem of algebraic combinatorics is to give a positive combinatorial formula for the structure constants of this basis. The celebrated Littlewood--Richardson rules solve this problem for special products $\sigma_u \cdot \sigma_v$ where $u$ and $v$ are $p$-Grassmannian permutations.

Building on work of Wyser, we introduce backstable clans to prove such a rule for the problem of computing the product $\sigma_u \cdot \sigma_v$ when $u$ is $p$-inverse Grassmannian and $v$ is $q$-inverse Grassmannian. By establishing several new families of linear relations among structure constants, we further extend this result to obtain a positive combinatorial rule for $\sigma_u \cdot \sigma_v$ in the case that $u$ is covered in weak Bruhat order by a $p$-inverse Grassmannian permutation and $v$ is a $q$-inverse Grassmannian permutation. 
\end{abstract}

\maketitle

\section{Introduction}\label{sec:intro}

The \newword{flag variety} $\Flags_n$ is the parameter space of complete nestings
\[
V_1 \subset V_2 \subset \dots \subset V_n = \mathbb{C}^n
\] of vector subspaces of $\mathbb{C}^n$, where $\dim V_i = i$. The Chow ring $A(\Flags_n)$ of the flag variety has a basis of \emph{Schubert cycles} $\sigma_u$ indexed by permutations $u \in S_n$. One of the major open problems of algebraic combinatorics is to give a positive combinatorial formula for the structure coefficients defined by 
\[
\sigma_u \cdot \sigma_v = \sum_w c_{u,v}^w \sigma_w.
\]
Such rules are currently known only for special classes of permutations. For example, the Littlewood--Richardson rules compute the coefficients in the cases where $u,v,w$ are all $p$-Grassmannian for some common $p$; here, we say $u$ is \newword{$p$-Grassmannian} if $u(i) < u(i+1)$ whenever $i \neq p$. For other special cases with known combinatorial rules, see, e.g., \cite{Monk, Sottile, Purbhoo.Sottile, Knutson.Purbhoo, Wyser, Meszaros.Panova.Postnikov, Buch.Kresch.Purbhoo.Tamvakis, Knutson.Zinn-Justin:1, Huang, Knutson.Zinn-Justin:3}. Much work has been done on extending the combinatorics of Littlewood--Richardson coefficients both to richer algebraic theories and to further families of analogous moduli spaces (see, e.g., \cite{Buch,Knutson.Tao:HT,Thomas.Yong:comin,Buch:twostep,Buch.Samuel,Pechenik.Yong}).  For discussion of the significance of positive combinatorial formulas for these and related numbers, see \cite{Knutson:ICM}.

Building on work of B.~Wyser \cite{Wyser}, we present an inverse Grassmannian analogue of the Littlewood--Richardson rule. Specifically, we solve the problem of giving a positive combinatorial formula for multiplying the Schubert cycles $\sigma_{u^{-1}}$ and $\sigma_{v^{-1}}$ where $u$ is $p$-Grassmannian and $v$ is $q$-Grassmannian. 
Wyser's work solved many instances of these problems, but required several additional technical hypotheses. We handle the remaining cases by embedding them into settings where these additional hypotheses hold. Wyser's approach was to realize $\sigma_{u^{-1}} \cdot \sigma_{v^{-1}}$ as the class of a \emph{Richardson variety} in $\Flags_n$ and to show that under his hypotheses this variety is the closure of an orbit for the action of a $2$-block Levi subgroup of $\GL_n(\CC)$; he then derives his formula from work of M.~Brion \cite{Brion:Korbit} describing the classes of such orbit-closures (\emph{$K$-orbits}). Our extension of Wyser's work, in contrast, involves purely combinatorial tools.

In sharp contrast to the Littlewood--Richardson case, these inverse Grassmannian products turn out to be \newword{multiplicity-free}, that is $c_{u^{-1}, v^{-1}}^w \in \{0,1\}$ for all $w$. Multiplicity-freeness has powerful geometric and combinatorial consequences (e.g., \cite{Brion,Knutson:Mobius,Hamaker.Patrias.Pechenik.Williams,Pechenik.Satriano}); for instance, our results imply that the $K$-theory classes of the corresponding Richardson varieties are determined by their Chow classes.

We now describe our first theorem in more detail.
Say a permutation $u$ is \newword{$p$-inverse Grassmannian} if $u^{-1}$ is $p$-Grassmannian and say $u$ is \newword{inverse Grassmannian} if it is $p$-inverse Grassmannian for some $p$.
Our combinatorial rule extends the rule of Wyser \cite[Theorem~3.10]{Wyser} based on combinatorial objects called \emph{clans}, which were introduced by \cite{Matsuki.Oshima, Yamamoto} in the context of $K$-orbits. Wyser's work provides a positive combinatorial formula for $c_{u, v}^w$ when \begin{itemize}
	\item $u$ and $v \in S_n$ are respectively $p$-inverse Grassmannian and $q$-inverse Grassmannian,
	\item $p+q=n$, and
	\item $u \leq w_0^{(n)}v$ (where $w_0^{(n)} \in S_n$ is the permutation of greatest \emph{Coxeter length} and the comparison is with respect to the \emph{strong Bruhat order}); 
\end{itemize}   
see Theorem~\ref{thm:wyser} for a precise statement of Wyser's theorem. 
To eliminate these technical conditions, we introduce the notion of \emph{backstable} clans by analogy with the backstable Schubert calculus of \cite{Lam.Lee.Shimozono}. We believe that backstable clans will additionally be amenable to the study of backstabilized $K$-orbits in infinite flag varieties; however, we do not pursue that application here. 

For inverse Grassmannian permutations $u,v$, we associate a backstable clan $\gamma_{u,v}$. (When Wyser's technical conditions hold, this backstable clan becomes the ordinary clan that he studies.) We also define a (backstable) \emph{rainbow clan} $\Omega_{p,q}$ associated to any pair of integers $p,q$. Finally, we need an action of the \emph{$0$-Hecke algebra} on backstable clans, denoted by ``$\cdot$''; we write $T_w$ for the element of the $0$-Hecke algebra corresponding to the permutation $w$.  All of these notions are defined precisely in Section~\ref{sec:background}. With these definitions, we have the following first main theorem, which we derive from Wyser's formula via stabilization arguments.

\begin{theorem}\label{thm:conj1}
 Let $u,v \in S_n$ be permutations, where
	 $u$ is $p$-inverse Grassmannian and $v$ is $q$-inverse Grassmannian. Then the product $\sigma_u \cdot \sigma_v \in A(\Flags_n)$ is a multiplicity-free sum of Schubert cycles. Precisely,
	 \[
	 \sigma_u \cdot \sigma_v = \sum_{w\in S_n} c_{u,v}^w \sigma_w,
	 \]
	 where \[
	 c_{u,v}^w = \begin{cases}
	 	1, & \text{if and only if } \ell(w)=\ell(u)+\ell(v) \text{ and }
	 T_w \cdot \gamma_{u,v} = \Omega_{p,q}; \\
	 0, & \text{otherwise}.
	 \end{cases} \]
\end{theorem}

For illustrations of the use of Theorem~\ref{thm:conj1}, see Examples~\ref{ex:231} and~\ref{ex:213-312}.

Our second main theorem uses analogous combinatorics to provide a positive combinatorial formula for a related class of products.
Let $s_i \in S_n$ denote the permutation that transposes $i$ and $i+1$.
Say that a permutation is \newword{subjacent} if it is of the form $s_p w$ for some $p$-inverse Grassmannian permutation $w \neq \id$.
Building on Theorem~\ref{thm:conj1}, we establish the following positive combinatorial rule for multiplying an inverse Grassmannian Schubert cycle by a subjacent Schubert cycle. Remarkably, such products are also multiplicity-free.

\begin{theorem}\label{thm:conj2}
Let $u,v \in S_n$ be permutations, where
	 $u\neq \id$ is $p$-inverse Grassmannian and $v$ is $q$-inverse Grassmannian. 
	 Then the product $\sigma_{s_p u} \cdot \sigma_v \in A(\Flags_n)$ is a multiplicity-free sum of Schubert cycles. Specifically,
	 \[
	 \sigma_{s_p u} \cdot \sigma_v = \sum_{w\in S_n} c_{s_p u,v}^w \sigma_w,
	 \]
	 where \[
	 c_{s_p u,v}^w = \begin{cases}
	 	1, & \text{if and only if } \ell(w)=\ell(u)+\ell(v)-1 \text{ and }
	 T_w \cdot \gamma_{u,v} \in \Psi_{p,q}; \\
	 0, & \text{otherwise};
	 \end{cases} \]
	 where $\Psi_{p,q}$ denotes a set of \emph{almost rainbow clans} defined precisely in Section~\ref{sec:background}.
\end{theorem}

For an illustration of the use of Theorem~\ref{thm:conj2}, see Example~\ref{ex:conj2}. Note that the product of two subjacent Schubert cycles is not generally multiplicity-free (see Example~\ref{ex:prodsubjacent}), so the multiplicity-freeness of Theorems~\ref{thm:conj1} and~\ref{thm:conj2} does not generalize.

Our main tools for deriving Theorem~\ref{thm:conj2} from Theorem~\ref{thm:conj1} are new families of linear relations among Schubert structure coefficients that we establish in  Propositions~\ref{prop:wlr} and~\ref{prop:stable}. We suspect that these linear relations have further consequences, which we briefly explore in Section~\ref{sec:WLR}.

This paper is organized as follows. In Section~\ref{sec:background}, we recall necessary background on permutations and Schubert polynomials, and we initiate a theory of backstable clans. In Section~\ref{sec:proof11}, we derive Theorem~\ref{thm:conj1} from a theorem of Wyser, together with the new notion of backstable clans and a well-known stabilization technique. In Section~\ref{sec:WLR}, we establish new linear relations among Schubert structure coefficients, together with $K$-theoretic analogues and various corollaries. In particular, the families of linear relations from Propositions~\ref{prop:wlr} and~\ref{prop:stable} will be key in our proof of Theorem~\ref{thm:conj2}. Section~\ref{sec:proof12} contains the proof of Theorem~\ref{thm:conj2} and related remarks.

\section{Preliminaries}\label{sec:background}

\subsection{Permutations}

We write $[n]$ for the set $\{1, 2, \dots, n\}$.

Let $S_\mathbb{Z}$ denote the group of permutations of $\mathbb{Z}$ that fix all but finitely many elements. For $i \in \mathbb{Z}$, the \newword{simple transposition} $s_i \in S_\mathbb{Z}$ is the involution that switches $i$ and $i+1$. Note that $S_\mathbb{Z}$ is generated by $\{s_i\}_{i \in \mathbb{Z}}$.  The \newword{(Coxeter) length} $\ell(w)$ of $w \in S_\mathbb{Z}$ is the length of a minimal expression for $w$ as a product of simple transpositions.

We write $S_+$ for the subgroup generated by $\{s_i\}_{i > 0}$ and write $S_n$ for the subgroup generated by $\{s_i\}_{0 < i < n}$.
For a permutation $w \in S_n$, we often write $w$ in one-line notation as $w(1)w(2) \dots w(n)$.  If we write $w$ in one-line notation, then $ws_i$ is obtained by swapping the entries in positions $i$ and $i+1$; $s_iw$ is obtained by swapping the letters $i$ and $i+1$.  The \newword{long element} $w_0^{(n)} \in S_n$ is $n (n-1) \dots 1$. The inclusion map $\iota : S_n \to S_{n+1}$ sends $w$ to $w(1)w(2)\cdots w(n)(n+1)$.

\newword{Left weak order} on permutations is defined by $u \leq_L w$ if $w = vu$ for some permutation $v$ with $ \ell(u) + \ell(v) = \ell(w)$. Similarly, in this case, we write $v \leq_R w$ and call this the \newword{right weak order}. Let $t_{i,j} \in S_{\mathbb{Z}}$ denote the involution swapping $i$ and $j$. \newword{Bruhat order} is the transitive closure of the covering relations $wt_{i,j} \lessdot w$ for $\ell(wt_{i,j}) = \ell(w) - 1$. We write Bruhat order comparisons as $u \leq w$, without subscripts. The weak orders are weak in the sense that the corresponding relations are subsets of the Bruhat order relation.

For a permutation $w \in S_\mathbb{Z}$, say that $i$ is a \newword{(right) descent} of $w$ if $w(i) > w(i+1)$, equivalently if $ws_i < w$.
Say that $i$ is a \newword{left descent} of $w \in S_\mathbb{Z}$ if $w^{-1}(i) > w^{-1}(i+1)$, equivalently if $s_i w < w$. 
The \newword{Lehmer code} of a permutation $w \in S_\mathbb{Z}$ is the function $c(w) : \mathbb{Z} \to \mathbb{Z}_{\geq 0}$  such that $c(w)(i)$ equals the number of $j > i$ such that $w(j) < w(i)$; as a shorthand, we often write $c_i(w) = c(w)(i)$.

A permutation $w \in S_\mathbb{Z}$ is \newword{$k$-Grassmannian} if $k$ is its unique descent or if it has no descent. Note that the identity permutation is $k$-Grassmannian for all $k$. We say $w$ is \newword{Grassmannian} if it is $k$-Grassmannian for some $k$. We say that $w$ is \newword{$k$-inverse Grassmannian} (resp.\ \newword{inverse Grassmannian}) if $w^{-1}$ is $k$-Grassmannian (resp.\ Grassmannian).

The \newword{$0$-Hecke algebra} $\mathcal{H}_\mathbb{Z}$ has generators $T_i$ for $i \in \mathbb{Z}$ satisfying 
\begin{align}\label{eq:braid}
\begin{split}
	 T_i^2 &= T_i, \\ 
	 T_i T_j &= T_j T_i \quad \text{(if $|i-j| > 1$), and} \\
	 T_i T_{i+1} T_i &= T_{i+1} T_i T_{i+1}. 
\end{split}
\end{align}
For every $w \in S_\mathbb{Z}$ of length $k$ there is a corresponding element $T_w \in \mathcal{H}_\mathbb{Z}$ obtained by taking any reduced decomposition $w = s_{i_1} \cdots s_{i_k}$ and setting $T_w = T_{i_1} \cdots T_{i_k}$. 
The elements $T_w$ for $w \in S_\mathbb{Z}$ are a linear basis of $\mathcal{H}_\mathbb{Z}$.

\subsection{Schubert polynomials}
\newword{Schubert polynomials} are defined recursively as follows. For $w_0^{(n)} \in S_n$, set the Schubert polynomial $\mathfrak{S}_{w_0^{(n)}} = x_1^{n-1} x_2^{n-2} \cdots x_n^0$. For $w$ such that $ws_i < w$, set 
\[
\mathfrak{S}_{ws_i} = N_i \mathfrak{S}_w,
\]
where $N_i$ is the \newword{(Newton) divided difference operator} that acts on $f \in \mathbb{Z}[x_1, \dots, x_n]$ by
\[
N_i(f)= \frac{f - s_i \cdot f}{x_i - x_{i+1}}.
\]
Here, $s_i$ acts on a polynomial by swapping variables $x_i$ and $x_{i+1}$. 

We have $\mathfrak{S}_w = \mathfrak{S}_{\iota(w)}$, so we may treat Schubert polynomials as indexed by the elements of $S_+$. The set of Schubert polynomials $\{\mathfrak{S}_w \}_{w \in S_+}$ is a linear basis of the free $\mathbb{Z}$-module $\mathbb{Z}[x_1, x_2, x_3, \dots]$. In particular, there are structure coefficients defined by 
\[
\mathfrak{S}_u \cdot \mathfrak{S}_v = \sum_w d_{u,v}^w \mathfrak{S}_w.
\]
For $u,v,w \in S_n$, these structure coefficients agree with the Schubert structure coefficients defined in Section~\ref{sec:intro}, i.e., $c_{u,v}^w = d_{u,v}^w$. Hence, we can study the structure coefficients $c_{u,v}^w$ by using Schubert polynomials in place of Schubert cycles. References for basic facts about Schubert polynomials include \cite{Macdonald:notes, Manivel}.

\subsection{Backstable clans}

A \newword{backstable clan} is a partial matching of the integers such that there exist $i,j \in \mathbb{Z}$ such that $i - k$ is paired with $j+k$ for all $k > 0$, together with an assignment of labels from $\{ \plus , \minus \}$ to the unmatched integers. We say that such a backstable clan $\gamma$ is \newword{supported} on $[i,j]$ and call $\gamma$ an \newword{$[i,j]$-clan}. (Note that if $\gamma$ is supported on $[i,j]$, then $\gamma$ is also supported on $[i-a,j+a]$ for any $a \geq 0$.) When we draw diagrams to illustrate an $[i,j]$-clan $\gamma$, we often restrict to the interval $[i,j]$, since all information about $\gamma$ can be extracted from this finite region. For examples of backstable clans, see Figures~\ref{fig:rainbow}, \ref{fig:heck_action}, and \ref{fig:almost_rainbow}.

In the previous literature, ``clans'' are restricted to the interval $[1,n]$; we identify these objects with $[1,n]$-clans. Clans were introduced in \cite{Matsuki.Oshima, Yamamoto} in the context of \emph{$K$-orbits}. For more recent work using clans in a related $K$-orbit context, see, e.g., \cite{Wyser.Yong,Woo.Wyser}. We believe that backstable clans will additionally be amenable to the study of backstabilized $K$-orbits, analogous to the backstable Schubert calculus of \cite{Lam.Lee.Shimozono}; however, we do not pursue that application here. In this paper, backstable clans are a tool for explicating the Schubert calculus of $\Flags_n$.

For a backstable clan $\gamma$, let $\zeta(\gamma)$ denote the number of $\plus$ labels minus the number of $\minus$ labels. If $\gamma$ is supported on $[i,j]$ and $\zeta(\gamma) = \zeta$, we say that $\gamma$ is a \newword{$(p,q)$-clan}, where $p = \frac{1}{2} (i+j+\zeta-1)$ and $q = \frac{1}{2}(i+j - \zeta-1)$. Note that $\zeta = p-q$ and $i+j-1 = p+q$; moreover, a backstable clan $\gamma$ supported on $[i,j]$ with $\zeta(\gamma) = p-q$ and $i+j-1 = p+q$ is a $(p,q)$-clan.

For a backstable clan $\gamma$, we write $\gamma(i) = j$ if $i$ is matched with $j$ in $\gamma$. 
Write $\sign$ to denote an unspecified element of $\{ \plus, \minus \}$. 
We write $\gamma(i) = \sign$ if $i$ is unmatched in $\gamma$ and labeled with $\sign \in \{\plus, \minus \}$.
If $i \in \mathbb{Z}$ is matched, we say that $i$ is \newword{initial} if $\gamma(i) > i$ and \newword{final} if $\gamma(i) < i$. A backstable clan $\gamma$ is \newword{noncrossing} if we never have $a<b<c<d \in \mathbb{Z}$ with $\gamma(a) = c$ and $\gamma(b) = d$.

For each pair of integers $p,q \in \mathbb{Z}$, the \newword{rainbow clan} $\Omega_{p,q}$ is the $(p,q)$-clan  such that
\[
\Omega_{p,q}(i) = \begin{cases}
	 \plus, & \text{if $i \in [q+1,p]$;} \\
	 \minus, & \text{if $i \in [p+1, q]$;} \\
	 p+q+1-i, & \text{otherwise}.\\
\end{cases}
\]
See Figure~\ref{fig:rainbow} for some examples. Note that the rainbow clan is always noncrossing and never has both $\plus$ and $\minus$ appearing.

%
%

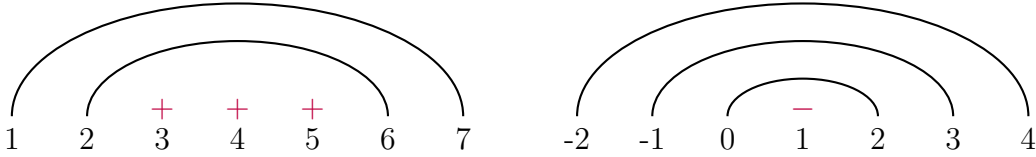
\begin{figure}[h]
\begin{center}
\begin{tikzpicture}
	\foreach \i in {1,...,7}
	{
	\node (\i) at (\i,0) {\i };	
	}
    \draw[thick] (7,0.3) arc
    [
        start angle=0,
        end angle=180,
        x radius=3cm,
        y radius =1.5cm
    ] ;
    \draw[thick] (6,0.3) arc
    [
        start angle=0,
        end angle=180,
        x radius=2cm,
        y radius =1cm
    ] ;
	\node (3lab) at (3,0.4) {$\plus$};
	\node (4lab) at (4,0.4) {$\plus$};
	\node (5lab) at (5,0.4) {$\plus$};
\end{tikzpicture}
\quad \quad
\begin{tikzpicture}
	\foreach \i in {-2,...,4}
	{
	\node (\i) at (\i+3,0) {\i };	
	}
    \draw[thick] (7,0.3) arc
    [
        start angle=0,
        end angle=180,
        x radius=3cm,
        y radius =1.5cm
    ] ;
    \draw[thick] (6,0.3) arc
    [
        start angle=0,
        end angle=180,
        x radius=2cm,
        y radius =1cm
    ] ;
    \draw[thick] (5,0.3) arc
    [
        start angle=0,
        end angle=180,
        x radius=1cm,
        y radius =0.5cm
    ] ;
	\node (4lab) at (4,0.4) {$\minus$};
\end{tikzpicture}
\end{center}
\caption{The rainbow clans $\Omega_{5,2}$ (left) and $\Omega_{0,1}$ (right).}
\label{fig:rainbow}
\end{figure}

For each generator $T_i$ of the $0$-Hecke algebra $\mathcal{H}_\mathbb{Z}$, we define an action of $T_i$ on $(p,q)$-clans. This action is defined through various cases; however, all have the flavor of acting locally at the numbers $i$ and $i+1$ and of transforming the $(p,q)$-clan to more closely resemble the rainbow clan $\Omega_{p,q}$. Precisely, for a clan $\gamma$, we have:
\begin{itemize}
	\item if $\gamma(i) = \sign$ and $i+1$ is initial, then \[
	(T_i \cdot \gamma)(i) = \gamma(i+1), (T_i \cdot \gamma)(\gamma(i+1))=i, \text{ and } (T_i \cdot \gamma)(i+1) = \sign;
	\]
	\item if $i$ is final and $\gamma(i+1) = \sign$, then 
	\[
	(T_i \cdot \gamma)(i) = \sign, (T_i \cdot \gamma)(\gamma(i))=i+1, \text{ and } (T_i \cdot \gamma)(i+1) = \gamma(i);
	\]
	\item if $i$ and $i+1$ are both initial with $\gamma(i) < \gamma(i+1)$, then \[
	(T_i \cdot \gamma)(i) = \gamma(i +1), (T_i \cdot \gamma)(\gamma(i +1)) = i, (T_i \cdot \gamma)(i+1) = \gamma(i), \text{ and } (T_i \cdot \gamma)(\gamma(i)) = i+1;
	\]
	\item if $i$ and $i+1$ are both final with $\gamma(i) < \gamma(i+1)$, then \[
	(T_i \cdot \gamma)(i) = \gamma(i +1), (T_i \cdot \gamma)(\gamma(i +1)) = i, (T_i \cdot \gamma)(i+1) = \gamma(i), \text{ and } (T_i \cdot \gamma)(\gamma(i)) = i+1;
	\]
	\item if $i$ is final and $i+1$ is initial, then 
	\[
	(T_i \cdot \gamma)(i) = \gamma(i +1), (T_i \cdot \gamma)(\gamma(i +1)) = i, (T_i \cdot \gamma)(i+1) = \gamma(i), \text{ and } (T_i \cdot \gamma)(\gamma(i)) = i+1;
	\]
	\item if $\gamma(i) = \sign$ and $\gamma(i+1) = \revsign$, then $(T_i \cdot \gamma)(i) = i+1$ and $(T_i \cdot \gamma)(i+1) = i$;
\end{itemize}
in all other cases, $T_i$ acts trivially. Since this action respects the braid relations of Equation~\eqref{eq:braid} by \cite[p.\ 839]{Wyser}, we obtain an action of each $0$-Hecke element $T_w$. (Wyser only considers $(p,q)$-clans supported on $[1,p+q]$; however, by translating $[i,j]$-clans to be supported on $[1,j-i+1]$, the general result is immediate.) Examples of the $0$-Hecke action on backstable clans are shown in Figure~\ref{fig:heck_action}.

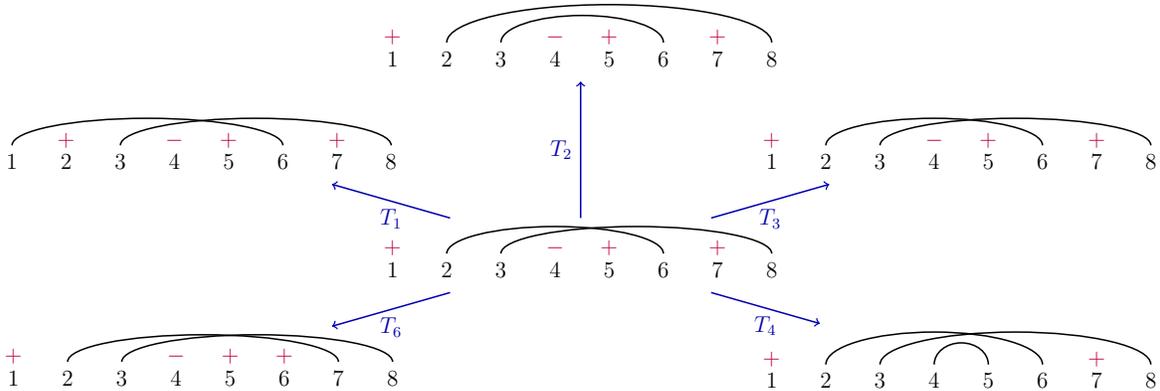
\begin{figure}[htb]
	\begin{center}
	\scalebox{0.72}{
\begin{tikzpicture}
	\node (center) at (0,0) {
		\begin{tikzpicture}
				\foreach \i in {1,...,8}
	{
	\node (\i) at (\i,0) {\i };	
	}
    \draw[thick] (6,0.3) arc
    [
        start angle=0,
        end angle=180,
        x radius=2cm,
        y radius =0.5cm
    ] ;
    \draw[thick] (8,0.3) arc
    [
        start angle=0,
        end angle=180,
        x radius=2.5cm,
        y radius =0.5cm
    ] ;
	\node (1lab) at (1,0.4) {$\plus$};
	\node (4lab) at (4,0.4) {$\minus$};
	\node (5lab) at (5,0.4) {$\plus$};
	\node (7lab) at (7,0.4) {$\plus$};
		\end{tikzpicture}};
	\node (s1) at (-7,2) {
		\begin{tikzpicture}
				\foreach \i in {1,...,8}
	{
	\node (\i) at (\i,0) {\i };	
	}
    \draw[thick] (6,0.3) arc
    [
        start angle=0,
        end angle=180,
        x radius=2.5cm,
        y radius =0.5cm
    ] ;
    \draw[thick] (8,0.3) arc
    [
        start angle=0,
        end angle=180,
        x radius=2.5cm,
        y radius =0.5cm
    ] ;
	\node (2lab) at (2,0.4) {$\plus$};
	\node (4lab) at (4,0.4) {$\minus$};
	\node (5lab) at (5,0.4) {$\plus$};
	\node (7lab) at (7,0.4) {$\plus$};
		\end{tikzpicture}};
    \node (s2) at (0,4) {
		\begin{tikzpicture}
				\foreach \i in {1,...,8}
	{
	\node (\i) at (\i,0) {\i };	
	}
    \draw[thick] (6,0.3) arc
    [
        start angle=0,
        end angle=180,
        x radius=1.5cm,
        y radius =0.5cm
    ] ;
    \draw[thick] (8,0.3) arc
    [
        start angle=0,
        end angle=180,
        x radius=3.0cm,
        y radius =0.7cm
    ] ;
	\node (1lab) at (1,0.4) {$\plus$};
	\node (4lab) at (4,0.4) {$\minus$};
	\node (5lab) at (5,0.4) {$\plus$};
	\node (7lab) at (7,0.4) {$\plus$};
		\end{tikzpicture}};
	\node (s3) at (7,2) {
		\begin{tikzpicture}
				\foreach \i in {1,...,8}
	{
	\node (\i) at (\i,0) {\i };	
	}
    \draw[thick] (6,0.3) arc
    [
        start angle=0,
        end angle=180,
        x radius=2cm,
        y radius =0.5cm
    ] ;
    \draw[thick] (8,0.3) arc
    [
        start angle=0,
        end angle=180,
        x radius=2.5cm,
        y radius =0.5cm
    ] ;
	\node (1lab) at (1,0.4) {$\plus$};
	\node (4lab) at (4,0.4) {$\minus$};
	\node (5lab) at (5,0.4) {$\plus$};
	\node (7lab) at (7,0.4) {$\plus$};
		\end{tikzpicture}};
	\node (s4) at (7,-2) {
		\begin{tikzpicture}
				\foreach \i in {1,...,8}
	{
	\node (\i) at (\i,0) {\i };	
	}
    \draw[thick] (6,0.3) arc
    [
        start angle=0,
        end angle=180,
        x radius=2cm,
        y radius =0.6cm
    ] ;
    \draw[thick] (8,0.3) arc
    [
        start angle=0,
        end angle=180,
        x radius=2.5cm,
        y radius =0.6cm
    ] ;
    \draw[thick] (5,0.3) arc
    [
        start angle=0,
        end angle=180,
        x radius=0.5cm,
        y radius =0.4cm
    ] ;
	\node (1lab) at (1,0.4) {$\plus$};
	\node (7lab) at (7,0.4) {$\plus$};
		\end{tikzpicture}};
	\node (s6) at (-7,-2) {
		\begin{tikzpicture}
				\foreach \i in {1,...,8}
	{
	\node (\i) at (\i,0) {\i };	
	}
    \draw[thick] (7,0.3) arc
    [
        start angle=0,
        end angle=180,
        x radius=2.5cm,
        y radius =0.5cm
    ] ;
    \draw[thick] (8,0.3) arc
    [
        start angle=0,
        end angle=180,
        x radius=2.5cm,
        y radius =0.5cm
    ] ;
	\node (1lab) at (1,0.4) {$\plus$};
	\node (4lab) at (4,0.4) {$\minus$};
	\node (5lab) at (5,0.4) {$\plus$};
	\node (6lab) at (6,0.4) {$\plus$};
		\end{tikzpicture}};
	\draw[->,thick,darkblue] (center) -- node[below]{$T_1$} (s1) ;
	\draw[->,thick,darkblue] (center) -- node[left]{$T_2$} (s2) ;
	\draw[->,thick,darkblue] (center) -- node[below]{$T_3$} (s3) ;
	\draw[->,thick,darkblue] (center) -- node[below]{$T_4$} (s4) ;
	\draw[->,thick,darkblue] (center) -- node[below]{$T_6$} (s6) ;
\end{tikzpicture}}
	\end{center}
	\caption{Some examples of the Hecke action on clans.}
	\label{fig:heck_action}
\end{figure}

We say that a $(p,q)$-clan $\gamma$ is an \newword{almost rainbow clan} if $T_i \cdot \gamma = \Omega_{p,q}$ for at least one $i \in \mathbb{Z}$. Examples of almost rainbow clans are depicted in Figure~\ref{fig:almost_rainbow}. Write $\omega_{p}$ for the almost rainbow $(p,p)$-clan with $\omega_p(p) = \minus$ and $\omega_p(p+1) = \plus$. For example, $\omega_3$ is illustrated in the center of Figure~\ref{fig:almost_rainbow}.

\begin{figure}[htb]
\scalebox{0.75}{
\begin{tikzpicture}
	\node (crossy) at (-7,0) {\begin{tikzpicture}
				\foreach \i in {3,...,9}
	{
	\node (\i) at (\i - 2,0) {\i };	
	}
    \draw[thick] (7,0.3) arc
    [
        start angle=0,
        end angle=180,
        x radius=3cm,
        y radius =0.8cm
    ] ;
    \draw[thick] (6,0.3) arc
    [
        start angle=0,
        end angle=180,
        x radius=1.5cm,
        y radius =0.5cm
    ] ;
    \draw[thick] (5,0.3) arc
    [
        start angle=0,
        end angle=180,
        x radius=1.5cm,
        y radius =0.5cm
    ] ;
	\node (4lab) at (4,0.4) {$\plus$};
		\end{tikzpicture}};
	\node (plusminus) at (0,0) {\begin{tikzpicture}
				\foreach \i in {1,...,6}
	{
	\node (\i) at (\i,0) {\i };	
	}
    \draw[thick] (6,0.3) arc
    [
        start angle=0,
        end angle=180,
        x radius=2.5cm,
        y radius =0.8cm
    ] ;
    \draw[thick] (5,0.3) arc
    [
        start angle=0,
        end angle=180,
        x radius=1.5cm,
        y radius =0.5cm
    ] ;
	\node (3lab) at (3,0.4) {$\minus$};
	\node (4lab) at (4,0.4) {$\plus$};
		\end{tikzpicture}};
	\node (shwoop) at (7,0) {\begin{tikzpicture}
				\foreach \i in {-4,...,2}
	{
	\node (\i) at (\i+5,0) {\i };	
	}
    \draw[thick] (7,0.3) arc
    [
        start angle=0,
        end angle=180,
        x radius=3cm,
        y radius =0.8cm
    ] ;
    \draw[thick] (5,0.3) arc
    [
        start angle=0,
        end angle=180,
        x radius=1.5cm,
        y radius =0.5cm
    ] ;
	\node (3lab) at (3,0.4) {$\minus$};
	\node (4lab) at (4,0.4) {$\minus$};
	\node (6lab) at (6,0.4) {$\minus$};
		\end{tikzpicture}};
	\end{tikzpicture}}
	\caption{Some representative almost rainbow clans.}
	\label{fig:almost_rainbow}
\end{figure}
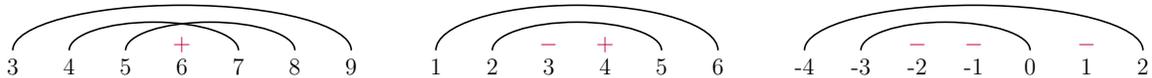

Suppose that $u \in S_\mathbb{Z}$ is $p$-inverse Grassmannian and $v \in S_\mathbb{Z}$ is $q$-inverse Grassmannian. We define a noncrossing backstable clan $\gamma_{u,v}$ associated to the pair $(u,v)$. Write $\hat{p} = p + \frac{1}{2}$ and $\hat{q}=q + \frac{1}{2}$. Then $\gamma_{u,v}$ is the unique noncrossing $(p,q)$-clan such that 
\begin{enumerate}
	\item if $u(i) < \hat{p}$ and $v(i) < \hat{q}$, then $i$ is initial; 	\item if $u(i) > \hat{p}$ and $v(i) > \hat{q}$, then $i$ is final; 	\item if $u(i) < \hat{p}$ and $v(i) > \hat{q}$,  then $\gamma_{u,v}(i) = \plus$;
and
	\item if $u(i) > \hat{p}$ and $v(i) < \hat{q}$, then $\gamma_{u,v}(i) = \minus$.
\end{enumerate}

See Figure~\ref{fig:uv_clan} for an example of this construction.

\begin{figure}[htb]
	\begin{center}
		\begin{tikzpicture}
				\foreach \i in {1,...,5}
	{
	\node (\i) at (\i,0) {\i };	
	}
    \draw[thick] (3,0.3) arc
    [
        start angle=0,
        end angle=180,
        x radius=1cm,
        y radius =0.5cm
    ] ;
    \draw[thick] (5,0.3) arc
    [
        start angle=0,
        end angle=180,
        x radius=0.5cm,
        y radius =0.5cm
    ] ;
	\node (2lab) at (2,0.4) {$\plus$};
		\end{tikzpicture}
	\end{center}
	\caption{The clan $\gamma_{u,v}$ for $u = 12435$ and $v = 13425$. Here, $p = 3$, $q = 2$, $\hat{p} = 3.5$, and $\hat{q} = 2.5$.}
	\label{fig:uv_clan}
\end{figure}
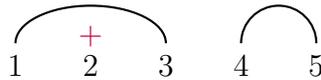

\begin{lemma}\label{lem:gamma_exists}
	The backstable clan $\gamma_{u,v}$ is well-defined.
\end{lemma}
\begin{proof}
	By assumption, there exists $N>0$ such that $u(i) = i$ and $v(i) = i$ for $i \geq N$ and for $i \leq -N$. Therefore, all $i<-N$ are initial and all $i>N$ are final. Since these sets are then both countably infinite, there is a unique noncrossing way to pair them up. Thus, $\gamma_{u,v}$ is a backstable clan; it remains to show that it is a $(p,q)$-clan.
	
	Choose an interval $[i,j]$ on which $\gamma_{u,v}$ is supported. Suppose $p \geq q$. By expanding the interval $[i,j]$ as necessary, assume that $i \leq q$ and $p \leq j$. 
	
	On the interval $[i,j]$, define
	\begin{align*}
		A &= \{z \in [i,j] : z \text{ is initial} \}, \\
		B &= \{z \in [i,j] : z \text{ is final} \}, \\
		C &=  \{z \in [i,j] : \gamma_{u,v}(z) = \plus \}, \text{ and} \\
		D &= \{z \in [i,j] : \gamma_{u,v}(z) = \minus \} 
	\end{align*}
	and define 
	\[
	a = |A|, b = |B|, c =|C|, \text{ and } c = |D|.
	\]
	On the interval $[i,j]$, both $u$ and $v$ take all of the values in $[i,j]$. We see that the $z \in [i,j]$ with $u(z) \in [i,p]$ are those $z \in A \cup C$, while those $z$ with $u(z) \in [p+1, j]$ are those $z \in B \cup D$. Therefore, $a+c = p-i+1$ and $b+d = j-p$. Similarly, the $z \in [i,j]$ with $v(z) \in [i,q]$ are those $z \in A \cup D$, while those $z$ with $v(z) \in [q+1,j]$ are those $z \in B \cup C$. Therefore, $a+d = q-i+1$ and $b+c = j-q$.
	
	By the definition of ``supported,'' the interval $[i,j]$ contains equal numbers of initial and of final elements, so $a=b$. Thus, $\zeta(\gamma_{u,v}) = c-d = j-q-(j-p) = p-q$, as desired for a $(p,q)$-clan.
	Moreover, we have $p-i+1 - (j-q) = 0$, so $p+q = i+j-1$, and so $\gamma_{u,v}$ is a $(p,q)$-clan.
	
	The case $p < q$ is entirely analogous; we omit the details.
\end{proof}

In the case that $p$ and $q$ are positive, $u,v \in S_{p+q}$, and $u \leq w_0^{(p+q)}v$, Lemma~\ref{lem:gamma_exists} was previously established by Wyser \cite[p.\ 840]{Wyser}.

For any positive integers $p,q \in \mathbb{Z}_{>0}$, we define a set $\Psi_{p,q}$ of almost rainbow $(p,q)$-clans. Let
\[\mbox{ \small  
$\Psi_{p,q} = \begin{cases}
	\{ \gamma : \gamma \text{ is almost rainbow and } T_i \cdot \gamma = \Omega_{p,q} \text{ for some }  i \in \mathbb{Z}_{>0} \text{ with } i \neq q \}, & \!\!\!\text{if } p \neq q; \\
	\{ \gamma : \gamma \text{ is almost rainbow and } T_i \cdot \gamma = \Omega_{p,q} \text{ for some } i \in \mathbb{Z}_{>0} \text{ with } i \neq q \} \cup \{\omega_p \}, & \!\!\! \text{if } p = q. \\
\end{cases}$}
\]

Note that in general $\Psi_{p,q} \neq \Psi_{q,p}$. If we relax the conditions $i \in \mathbb{Z}_{>0}$ to $i \in \mathbb{Z}$, then the enlarged set of almost rainbow $(p,q)$-clans can be used to compute \emph{backstable Schubert structure coefficients} (introduced in \cite{Lam.Lee.Shimozono}).

\section{Discussion and proof of \texorpdfstring{Theorem~\ref{thm:conj1}}{Theorem 1.1}}\label{sec:proof11}
First, we recall \cite[Theorem~3.10]{Wyser}, which establishes the case of Theorem~\ref{thm:conj1} where $u,v,w \in S_{p+q}$ and $u \leq w_0^{(p+q)}v$. We will use Wyser's result to prove Theorem~\ref{thm:conj1}. Theorem~\ref{thm:conj1} will then be a major ingredient in our proof of Theorem~\ref{thm:conj2} in Section~\ref{sec:proof12}.

\begin{theorem}[{\cite[Theorem~3.10]{Wyser}}]\label{thm:wyser}
	Let $u,v,w \in S_{p+q}$ be permutations, where
	 $u$ is $p$-inverse Grassmannian, $v$ is $q$-inverse Grassmannian, and $u \leq w_0^{(p+q)}v$. Then $c_{u,v}^w \in \{0,1\}$. Moreover, we have $c_{u,v}^w = 1$ if and only if $\ell(w)=\ell(u)+\ell(v)$ and
	 \pushQED{\qed}
	 \[
	 T_w \cdot \gamma_{u,v} = \Omega_{p,q}.  \qedhere \popQED\] \let\qed\relax
\end{theorem}

For a permutation $w \in S_n$, let $1 \times w \in S_{n+1}$ denote the permutation such that $(1 \times w)(1) = 1$ and $(1 \times w)(i) = w(i-1)+1$ for $i>1$. Iterating this operation gives rise to the notion of \emph{backstabilization} of Schubert calculus; for further discussion, see \cite{Lam.Lee.Shimozono,Nenashev}. We will, however, only need to apply this operation once. The following fact is straightforward; moreover, it is a special case of Lemma~\ref{lem:2}, which we prove later.

\begin{lemma}\label{lem:easy}
	For any $u,v,w \in S_+$, we have $c_{u,v}^w=c_{1\times u,1\times v}^{1\times w}$. \qed 
\end{lemma}


\begin{proof}[Proof of Theorem~\ref{thm:conj1}]
	Fix an interval $[i,j]$ on which $\gamma_{u,v}$ is supported. If $i >1$, then $\gamma_{u,v}$ is also supported on the interval $[1,p+q]$. Therefore, $\gamma_{u,v}$ is a noncrossing $[1,p+q]$-clan, so by \cite[Remark~3.9]{Wyser}, we have $u, v \in S_{p+q}$ and $u \leq w_0^{(p+q)} v$. We are then done in this case by Theorem~\ref{thm:wyser}.
	
	Suppose instead that $i < 1$. Then, let $\tilde{\gamma}$ be the horizontal shift of $\gamma_{u,v}$ to the right by $1-i$. That is, let 
	\[
	\tilde{\gamma}(z) = \gamma_{u,v}(z+i-1),
	\]
	so that $\tilde{\gamma}$ is supported on $[1, j-i+1]$.
	Also define $\tilde{u}$ and $\tilde{v}$ by
	\[
	\tilde{u}(z) = u(z+i-1) \text{ and } \tilde{v}(z) = v(z+i-1).
	\]
	Then, $\tilde{u}$ is $\tilde{p}$-inverse Grassmannian and $\tilde{v}$ is $\tilde{q}$-inverse Grassmannian, where $\tilde{p} = p+1-i$ and $\tilde{q} = q+1-i$. Note that, since $p+q = i+j-1$, we have $\tilde{p} + \tilde{q} = j-i+1$.
	Further, observe that $\gamma_{\tilde{u}, \tilde{v}} = \tilde{\gamma}$ by construction. 
	
	The clan $\gamma_{\tilde{u}, \tilde{v}}$ is a noncrossing clan supported on $[1,\tilde{p} + \tilde{q}]$. Therefore, by \cite[Remark~3.9]{Wyser}, we have $\tilde{u}, \tilde{v} \in S_{\tilde{p} + \tilde{q}}$ and $\tilde{u} \leq w_0^{(\tilde{p} + \tilde{q})} \tilde{v}$.
	
	Now, define $\tilde{w}$ by 
	\[
	\tilde{w}(z) = w(z+i-1).
	\]
	Note that 
	\[ 
	(T_w \cdot \gamma_{u,v})(z) = (T_{\tilde{w}} \cdot \gamma_{\tilde{u}, \tilde{v}})(z+i -1)
\]
for all $z$. This implies that $T_w \cdot \gamma_{u,v} = \Omega_{p,q}$ if and only if $T_{\tilde{w}} \cdot \gamma_{\tilde{u}, \tilde{v}} = \Omega_{\tilde{p}, \tilde{q}}$. 

By Lemma~\ref{lem:easy}, we have that
\[
c_{u,v}^w = c_{\tilde{u}, \tilde{v}}^{\tilde{w}},
\]
since $\tilde{u} = 1^{1-i} \times u$, $\tilde{v} = 1^{1-i} \times v$, and $\tilde{w} = 1^{1-i} \times w$. Thus, by Theorem~\ref{thm:wyser}, we have $c_{u,v}^w \in \{0,1\}$. Also note that $\ell(\tilde{u}) = \ell(u)$, $\ell(\tilde{v}) = \ell(v)$, and $\ell(\tilde{w}) = \ell(w)$. 
Thus, Theorem~\ref{thm:wyser} additionally yields that $c_{u,v}^w = 1$ if and only if $T_w \cdot \gamma_{u,v} = \Omega_{p,q}$. This completes the proof of Theorem~\ref{thm:conj1}.
\end{proof}

\begin{remark}\label{rem:cant_do_it}
	Note that the hypotheses of Theorem~\ref{thm:wyser} are somewhat restrictive.  For instance, Theorem~\ref{thm:wyser} is unable to compute any structure coefficient of the form $c_{231,231}^w$, since $231$ is $1$-inverse Grassmannian but $231\not \in S_{1+1}$.  Example~\ref{ex:231} demonstrates how we can instead compute these structure coefficients using backstable clans and Theorem~\ref{thm:conj1}.
	
	Similarly, Theorem~\ref{thm:wyser} cannot compute any of the structure coefficients $c_{213,312}^w$ because $213\not \leq w_{0}^{(3)}312=132$. See Example~\ref{ex:213-312} for a demonstration of computing these structure coefficients through backstable clans and Theorem~\ref{thm:conj1}.
\end{remark}

We now show how Theorem~\ref{thm:conj1} uses backstable clans to compute the Schubert structure coefficients described in Remark~\ref{rem:cant_do_it}.

\begin{example}\label{ex:231}
	Let $u = 231 \in S_3$. The backstable clan $\gamma_{u,u}$ looks like
	\begin{center}
\begin{tikzpicture}
				\foreach \i in {-1,...,4}
	{
	\node (\i) at (\i,0) {\i };	
	}
    \draw[thick] (1,0.3) arc
    [
        start angle=0,
        end angle=180,
        x radius=.5cm,
        y radius =0.4cm
    ] ;
        \draw[thick] (2,0.3) arc
    [
        start angle=0,
        end angle=180,
        x radius=1.5cm,
        y radius =0.7cm
    ] ;
    \draw[thick] (4,0.3) arc
    [
        start angle=0,
        end angle=180,
        x radius=0.5cm,
        y radius =0.4cm
    ] ;
		\end{tikzpicture}.
	\end{center} 
	We consider all nontrivial actions of $0$-Hecke generators $T_i$ on $\gamma_{u,u}$, until reaching the rainbow clan $\Omega_{1,1}$:
		\begin{center}
	\scalebox{0.72}{
\begin{tikzpicture}
	\node (center) at (0,0) {
		\begin{tikzpicture}
	\foreach \i in {-1,...,4}
	{
		\node (\i) at (\i,0) {\i };	
	}
	\draw[thick] (1,0.3) arc
	[
	start angle=0,
	end angle=180,
	x radius=.5cm,
	y radius =0.4cm
	] ;
	\draw[thick] (3,0.3) arc
	[
	start angle=0,
	end angle=180,
	x radius=2cm,
	y radius =0.75cm
	] ;
	\draw[thick] (4,0.3) arc
	[
	start angle=0,
	end angle=180,
	x radius=1cm,
	y radius =0.4cm
	] ;
	
\end{tikzpicture}};
    \node (s2) at (0,3) {
			\begin{tikzpicture}
				\foreach \i in {-1,...,4}
	{
	\node (\i) at (\i,0) {\i };	
	}
    \draw[thick] (1,0.3) arc
    [
        start angle=0,
        end angle=180,
        x radius=.5cm,
        y radius =0.4cm
    ] ;
        \draw[thick] (2,0.3) arc
    [
        start angle=0,
        end angle=180,
        x radius=1.5cm,
        y radius =0.7cm
    ] ;
    \draw[thick] (4,0.3) arc
    [
        start angle=0,
        end angle=180,
        x radius=0.5cm,
        y radius =0.4cm
    ] ;
		\end{tikzpicture}};
	\node (s4) at (7,-2) {
		\begin{tikzpicture}
	\foreach \i in {-1,...,4}
	{
		\node (\i) at (\i,0) {\i };	
	}
	\draw[thick] (1,0.3) arc
	[
	start angle=0,
	end angle=180,
	x radius=.5cm,
	y radius =0.4cm
	] ;
	\draw[thick] (4,0.3) arc
	[
	start angle=0,
	end angle=180,
	x radius=2.5cm,
	y radius =0.75cm
	] ;
	\draw[thick] (3,0.3) arc
	[
	start angle=0,
	end angle=180,
	x radius=.5cm,
	y radius =0.4cm
	] ;
	
\end{tikzpicture}};
	\node (s6) at (-7,-2) {
		\begin{tikzpicture}
	\foreach \i in {-1,...,4}
	{
		\node (\i) at (\i,0) {\i };	
	}
	\draw[thick] (2,0.3) arc
	[
	start angle=0,
	end angle=180,
	x radius=1cm,
	y radius =0.5cm
	] ;
	\draw[thick] (3,0.3) arc
	[
	start angle=0,
	end angle=180,
	x radius=2cm,
	y radius =0.75cm
	] ;
	\draw[thick] (4,0.3) arc
	[
	start angle=0,
	end angle=180,
	x radius=1.5cm,
	y radius =0.5cm
	] ;
\end{tikzpicture}
};
	\node (rainbow) at (0,-5) {
		\begin{tikzpicture}
	\foreach \i in {-1,...,4}
	{
		\node (\i) at (\i,0) {\i };	
	}
	\draw[thick] (2,0.3) arc
	[
	start angle=0,
	end angle=180,
	x radius=1cm,
	y radius =0.5cm
	] ;
	\draw[thick] (4,0.3) arc
	[
	start angle=0,
	end angle=180,
	x radius=2.5cm,
	y radius =0.75cm
	] ;
	\draw[thick] (3,0.3) arc
	[
	start angle=0,
	end angle=180,
	x radius=1cm,
	y radius =0.5cm
	] ;
\end{tikzpicture}};
	\node (leftthing) at (-8,-5) {
							\begin{tikzpicture}
	\foreach \i in {-1,...,4}
	{
		\node (\i) at (\i,0) {\i };	
	}
	\draw[thick] (2,0.3) arc
	[
	start angle=0,
	end angle=180,
	x radius=.5cm,
	y radius =0.4cm
	] ;
	\draw[thick] (3,0.3) arc
	[
	start angle=0,
	end angle=180,
	x radius=2cm,
	y radius =0.7cm
	] ;
	\draw[thick] (4,0.3) arc
	[
	start angle=0,
	end angle=180,
	x radius=2cm,
	y radius =0.7cm
	] ;
\end{tikzpicture}
};
\node (actualrainbow) at (-4,-8) {
							
\begin{tikzpicture}
	\foreach \i in {-1,...,4}
	{
		\node (\i) at (\i,0) {\i };	
	}
	\draw[thick] (3,0.3) arc
	[
	start angle=0,
	end angle=180,
	x radius=1.5cm,
	y radius =0.75cm
	] ;
	\draw[thick] (4,0.3) arc
	[
	start angle=0,
	end angle=180,
	x radius=2.5cm,
	y radius =1cm
	] ;
	\draw[thick] (2,0.3) arc
	[
	start angle=0,
	end angle=180,
	x radius=.5cm,
	y radius =0.4cm
	] ;
\end{tikzpicture}
};
	\draw[<-,thick,darkblue] (center) -- node[left]{$T_2$} (s2) ;
	\draw[->,thick,darkblue] (center) -- node[above right]{$T_3$} (s4) ;
	\draw[->,thick,darkblue] (center) -- node[above left]{$T_1$} (s6) ;
	\draw[->,thick,darkblue] (s4) -- node[below right]{$T_1$} (rainbow) ;
	\draw[->,thick,darkblue] (s6) -- node[below]{$T_3$} (rainbow) ;
	\draw[->,thick,darkblue] (rainbow) -- node[below right]{$T_{0},T_2$} (actualrainbow) ;
	\draw[->,thick,darkblue] (s6) -- node[above left]{$T_0$} (leftthing) ;
	\draw[->,thick,darkblue] (leftthing) -- node[below left]{$T_{-1},T_3$} (actualrainbow) ;
\end{tikzpicture}}.
	\end{center}

There are two paths in this diagram from $\gamma_{u,u}$ to $\Omega_{1,1}$ using only $T_i$ with $i>0$; these paths are
labeled by the sequences $(T_2,T_1,T_3,T_2)$ and $(T_2,T_3,T_1,T_2)$, which both correspond to the permutation $3412 = s_2 s_3 s_1 s_2 = s_2 s_1 s_3 s_2$.  Thus by Theorem~\ref{thm:conj1}, we have that $c_{u,u}^{3412} = 1$, while $c_{u,u}^w = 0$ for all $w \neq 3412$.
\end{example}

\begin{example}\label{ex:213-312}
	Let $u = 213$ and $v = 312$. Then the backstable clan $\gamma_{u,v}$ looks like
		\begin{center}
		\begin{tikzpicture}
				\foreach \i in {0,...,4}
	{
	\node (\i) at (\i+1,0) {\i };	
	}
    \draw[thick] (2,0.3) arc
    [
        start angle=0,
        end angle=180,
        x radius=0.5cm,
        y radius =0.5cm
    ] ;
    \draw[thick] (5,0.3) arc
    [
        start angle=0,
        end angle=180,
        x radius=1.0cm,
        y radius =0.5cm
    ] ;
	\node (2lab) at (4,0.4) {$\minus$};
		\end{tikzpicture}.
	\end{center}
	We consider all nontrivial actions of $0$-Hecke generators $T_i$ on $\gamma_{u,v}$, until reaching the rainbow clan $\Omega_{1,2}$:
		\begin{center}
	\scalebox{0.72}{
\begin{tikzpicture}
	\node (center) at (0,0) {
		\begin{tikzpicture}
				\foreach \i in {0,...,4}
	{
	\node (\i) at (\i+1,0) {\i };	
	}
    \draw[thick] (5,0.3) arc
    [
        start angle=0,
        end angle=180,
        x radius=1.5cm,
        y radius =0.5cm
    ] ;
    \draw[thick] (3,0.3) arc
    [
        start angle=0,
        end angle=180,
        x radius=1.0cm,
        y radius =0.5cm
    ] ;
	\node (4lab) at (4,0.4) {$\minus$};
		\end{tikzpicture}};
    \node (s2) at (0,3) {
			\begin{tikzpicture}
				\foreach \i in {0,...,4}
	{
	\node (\i) at (\i+1,0) {\i };	
	}
    \draw[thick] (2,0.3) arc
    [
        start angle=0,
        end angle=180,
        x radius=0.5cm,
        y radius =0.5cm
    ] ;
    \draw[thick] (5,0.3) arc
    [
        start angle=0,
        end angle=180,
        x radius=1.0cm,
        y radius =0.5cm
    ] ;
	\node (2lab) at (4,0.4) {$\minus$};
		\end{tikzpicture}};
	\node (s4) at (7,-2) {
		\begin{tikzpicture}
				\foreach \i in {0,...,4}
	{
	\node (\i) at (\i+1,0) {\i };	
	}
    \draw[thick] (5,0.3) arc
    [
        start angle=0,
        end angle=180,
        x radius=1.5cm,
        y radius =0.6cm
    ] ;
    \draw[thick] (4,0.3) arc
    [
        start angle=0,
        end angle=180,
        x radius=1.5cm,
        y radius =0.6cm
    ] ;
	\node (1lab) at (3,0.4) {$\minus$};
		\end{tikzpicture}};
	\node (s6) at (-7,-2) {
		\begin{tikzpicture}
				\foreach \i in {0,...,4}
	{
	\node (\i) at (\i+1,0) {\i };	
	}
    \draw[thick] (3,0.3) arc
    [
        start angle=0,
        end angle=180,
        x radius=0.5cm,
        y radius =0.4cm
    ] ;
    \draw[thick] (5,0.3) arc
    [
        start angle=0,
        end angle=180,
        x radius=2.0cm,
        y radius =0.7cm
    ] ;
	\node (4lab) at (4,0.4) {$\minus$};
		\end{tikzpicture}};
	\node (rainbow) at (0,-4) {
		\begin{tikzpicture}
				\foreach \i in {0,...,4}
	{
	\node (\i) at (\i+1,0) {\i };	
	}
    \draw[thick] (5,0.3) arc
    [
        start angle=0,
        end angle=180,
        x radius=2.0cm,
        y radius =0.7cm
    ] ;
    \draw[thick] (4,0.3) arc
    [
        start angle=0,
        end angle=180,
        x radius=1.0cm,
        y radius =0.5cm
    ] ;
	\node (1lab) at (3,0.4) {$\minus$};
		\end{tikzpicture}};
	\draw[<-,thick,darkblue] (center) -- node[left]{$T_1$} (s2) ;
	\draw[->,thick,darkblue] (center) -- node[below]{$T_2$} (s4) ;
	\draw[->,thick,darkblue] (center) -- node[below]{$T_0$} (s6) ;
	\draw[->,thick,darkblue] (s4) -- node[below right]{$T_0, T_3$} (rainbow) ;
	\draw[->,thick,darkblue] (s6) -- node[below]{$T_2$} (rainbow) ;
\end{tikzpicture}}.
	\end{center}
	There is a unique path in this diagram from $\gamma_{u,v}$ to $\Omega_{1,2}$ using only $T_i$ with $i>0$, namely that labeled by the sequence $(T_1, T_2, T_3)$. Note that $s_3 s_2 s_1 = 4123$. We conclude by Theorem~\ref{thm:conj1} that $c_{u,v}^{4123} = 1$, while $c_{u,v}^w = 0$ for all $w \neq 4123$.
\end{example}

\section{Linear relations among Schubert structure coefficients}\label{sec:WLR}

In this section, we establish new linear relations among Schubert structure coefficients. In the first subsection, we derive linear relations among cohomological structure coefficients; we will use these relations in Section~\ref{sec:proof12} to prove Theorem~\ref{thm:conj2}. In the second subsection, we derive analogous linear relations among $K$-theoretic structure coefficients; these relations will not be explored further in the later sections of this paper. The third subsection studies stabilization phenomena to obtain additional linear relations among cohomological structure coefficients; these relations will also be important to the proof of Theorem~\ref{thm:conj2} in Section~\ref{sec:proof12}. The fourth subsection considers relations obtained by iterating the technique of the first subsection; these relations will not be studied further in this paper.

\subsection{Cohomology}

We will need the differential operator $\nabla : \mathbb{Z}[x_1, x_2, \dots] \to \mathbb{Z}[x_1, x_2, \dots]$ defined by 
\[
\nabla = \sum_{i=1}^\infty \frac{\partial}{\partial x_i}.
\]
Our key tool will be the following, developed earlier in our joint work with Z.~Hamaker and D.~Speyer.

\begin{proposition}[{\cite[Proposition~1.1]{Hamaker.Pechenik.Speyer.Weigandt}}]\label{prop:HPSW}
	For $w \in S_+$, we have 
	\[
	\nabla \mathfrak{S}_w = \sum_{s_k w < w} k \mathfrak{S}_{s_k w}.
	\]
\end{proposition}

From this proposition, we can establish our primary family of linear relations. 

\begin{proposition}
\label{prop:wlr}
	Let $u,v,w \in S_+$. Then
	\begin{equation}\label{eq:main}
	\sum_{s_iu < u} i c_{s_iu,v}^w + \sum_{s_jv < v} j c_{u,s_jv}^w = \sum_{s_kw> w} k c_{u,v}^{s_k w}.
	\end{equation}
\end{proposition}

Proposition~\ref{prop:wlr} enables one to discern properties of an unknown $c_{u,v}^w$ from properties of other Schubert structure coefficients. In particular, our relations yield some new vanishing and nonvanishing conditions, as well as congruence conditions.  We do not know how to relate our linear relations to the the nonpositive recurrence of \cite{Knutson:recurrence}.

\begin{proof}[Proof of Proposition~\ref{prop:wlr}]
	Write \[
	\mathfrak{S}_u \cdot \mathfrak{S}_v = \sum_p c_{u,v}^p \mathfrak{S}_p
	\]
	 and apply the differential operator $\nabla$ to both sides to obtain
\begin{equation}\label{eq:nabla_the_product}
	\sum_{s_i u < u} i \mathfrak{S}_{s_i u} \mathfrak{S}_v + \sum_{s_j v < v} j \mathfrak{S}_u \mathfrak{S}_{s_j v} = \sum_p \sum_{s_k p < p} k c_{u,v}^p \mathfrak{S}_{s_k p}
\end{equation}
	by Proposition~\ref{prop:HPSW} and the Leibniz formula. Now extract the coefficient of $\mathfrak{S}_w$ from both sides of Equation~\eqref{eq:nabla_the_product} to obtain
	\[
	\sum_{s_i u < u} i c_{s_i u, v}^w + \sum_{s_jv < v} j c_{u,s_jv}^w = \sum_{s_k w > w} k c_{u,v}^{s_k w},
	\]
	as desired.
\end{proof}

Proposition~\ref{prop:wlr} has some surprising corollaries. The following result can be extracted straightforwardly from Monk's formula \cite{Monk}, but we can alternatively derive it easily from Proposition~\ref{prop:wlr}.

\begin{corollary}\label{cor:monk}
	Let $v \in S_+$ and let $i \in \mathbb{Z}_+$. Then there is some $k \in \mathbb{Z}_+$ such that $c_{s_i, v}^{s_k v} > 0$.
\end{corollary}
\begin{proof}
	Specialize Proposition~\ref{prop:wlr} to the case $u = s_i$ and $w=v$. Then we get
	\[
	i + \sum_{s_j v < v} j c_{s_i, s_jv}^w = \sum_{s_k v > v} k c_{s_i, v}^{s_k v}.
	\]
	Since the sum on the left is nonnegative and $i>0$, we obtain
	\[
	0 < \sum_{s_k v > v} k c_{s_i, v}^{s_k v}.
	\]
	But then 
	\[
	0 < \sum_{s_k v > v} c_{s_i, v}^{s_k v},
	\]
so there is some $k$ with $s_k v > v$ and $c_{s_i, v}^{s_k v} > 0$. By dimension counting, the second of these conditions implies the first, so the corollary follows.
\end{proof}

Proposition~\ref{prop:wlr} also implies many congruence relations among Schubert structure coefficients.

\begin{corollary}\label{cor:residue}
	Suppose all left descents of $u$ and $v$ are multiples of $\alpha\in \mathbb Z_+$. Then, for any $w \in S_+$, we have
	\[
	\sum_{s_k w > w} k c_{u,v}^{s_k w} \equiv 0 \pmod{\alpha}.
	\]
	If moreover $u=v$, then 
	\[
	\sum_{s_k w > w} k c_{u,u}^{s_k w} \equiv 0 \pmod {2\alpha}.
	\]
\end{corollary}
\begin{proof}
Under the hypotheses of the corollary, every term of each sum on the left of Equation~\eqref{eq:main} is a multiple of $\alpha$. Hence, the sum on the right is as well. If also $u=v$, then the two sums on the left of Equation~\eqref{eq:main} are equal to each other.
\end{proof}

What is remarkable about Corollary~\ref{cor:residue} is that although our sum is congruent to $0$ modulo $\alpha$, the individual terms of the sum generally are not. For this reason, knowing some of the relevant Schubert structure coefficients imposes strong conditions on the remaining ones. Before giving an example of the application of Corollary~\ref{cor:residue}, we need the following easy lemma.

\begin{lemma}\label{lem:easy_vanish}
	Suppose $u,v,a \in S_n$ and let $m >n$. Then $c_{u,v}^{s_m a} = 0$.
\end{lemma}
\begin{proof}
	The Schubert polynomials $\mathfrak{S}_u$, $\mathfrak{S}_v$, $\mathfrak{S}_a$ all lie in $\mathbb{Z}[x_1, \dots, x_{n-1}]$. On the other hand, $m$ is a (right) descent of $s_m a$, so $\mathfrak{S}_{s_m a}$ involves the variable $x_m$, so $\mathfrak{S}_{s_m a} \notin \mathbb{Z}[x_1, \dots, x_{n-1}]$. Hence, $c_{u,v}^{s_m a} = 0$.
\end{proof}

\begin{example}\label{ex:modular_arithmetic}
Suppose $u=13254$ and note that it only has left descents $2$ and $4$. Let $w = 231645$ and suppose we have correctly computed already that $c_{u,u}^w = 1$. 

Now let $a = s_1w = 132645$. Corollary~\ref{cor:residue} gives that 
\begin{equation}\label{eq:a}
	\sum_{s_k a > a} k c_{u,u}^{s_k a} \equiv 0 \pmod{4}.
\end{equation}
But we can expand this sum as 
\[
\sum_{s_k a > a} k c_{u,u}^{s_k a} = 1 c_{u,u}^{s_1 a}  +3 c_{u,u}^{s_3 a} +4 c_{u,u}^{s_4 a} +6 c_{u,u}^{s_6 a},\]
using Lemma~\ref{lem:easy_vanish} to see that the other potential terms vanish.
Since we know the term with coefficient $1$ contributes $1$, the sum of the other terms must be congruent to $3$ modulo $4$. In particular, we learn for free that  $c_{u,u}^{s_3 a} \neq 0$. 
 Even better, it is immediate without further computation that $c_{u,u}^{s_3 a}$ is odd. In fact, it turns out that $c_{u,u}^{s_3 a} = 1$.
 
If we compute also that $c_{u,u}^{s_3 a} = 1$, we learn then that $c_{u,u}^{s_6 a}$ must be even. In fact, it turns out that $c_{u,u}^{s_6 a} = 0$.
\end{example}

\subsection{$K$-theory} 

The structure sheaves of Schubert varieties $X_w \subset \Flags_n$ give classes $[\mathcal{O}_{X_w}]$ in the Grothendieck ring $K^0(\Flags_n)$ of algebraic vector bundles over $\Flags_n$. These classes form an additive basis and give rise to $K$-theoretic Schubert structure coefficients $K_{u,v}^w$ defined by
\[
[\mathcal{O}_{X_u}] \cdot [\mathcal{O}_{X_v}] = \sum_w K_{u,v}^w [\mathcal{O}_{X_w}].
\] When $\ell(w) = \ell(u) + \ell(v)$, we have $K_{u,v}^w = c_{u,v}^w$, but, unlike $c_{u,v}^w$, $K_{u,v}^w$ can be nonzero when $\ell(w) > \ell(u) + \ell(v)$.

\newword{Grothendieck polynomials} $\mathfrak{G}_w$ represent Schubert structure sheaf classes analogously to how Schubert polynomials represent Schubert cycles. We may also define Grothendieck polynomials recursively. For $w_0^{(n)} \in S_n$, we set $\mathfrak{G}_{w_0^{(n)}} = \mathfrak{S}_{w_0^{(n)}} = x_1^{n-1}x_2^{n-2}\cdots x_n^0$.  For $w$ such that $ws_i < w$, set
\[
\mathfrak{G}_{ws_i} = \overline{N}_i \mathfrak{G}_w,
\]
where $\overline{N}_i(f) = N_i( (1-x_{i+1}) f)$. We have $\mathfrak{G}_w = \mathfrak{G}_{\iota(w)}$, so we think of Grothendieck polynomials as also being indexed by elements of $S_+$.  The set of Grothendieck polynomials $\{\mathfrak{G}_w \}_{w \in S_+}$ is another linear basis of $\mathbb{Z}[x_1, x_2, x_3, \dots]$. The structure coefficients defined by 
\[
\mathfrak{G}_u \cdot \mathfrak{G}_v = \sum_w L_{u,v}^w \mathfrak{G}_w
\]
agree with the $K$-theoretic Schubert structure coefficients $K_{u,v}^w$ provided $u,v,w \in S_n$. 

Let $\beta$ be an indeterminate. Define the \newword{$\beta$-Grothendieck polynomial} $\mathfrak{G}_w^{(\beta)}$ by
\[
\mathfrak{G}_w^{(\beta)}(x_1, \dots, x_n) = (-\beta)^{-\ell(w)} \mathfrak{G}_w(-\beta x_1,  \dots, -\beta x_n). 
\] The $\beta$-Grothendieck polynomials were introduced in \cite{Fomin.Kirillov} and represent classes in the \emph{connective $K$-theory} of $\Flags_n$ \cite{Hudson}. We will find the $\beta$-Grothendieck polynomials slightly easier to work with in our context. Let the structure coefficients for $\beta$-Grothendieck polynomials be $K_{u,v}^w(\beta)$. We have $K_{u,v}^w(-1) = K_{u,v}^w$.

Let $\Des(w)$ denote the set of descents of the permutation $w$. 
 The \newword{major index} of $w$ is 
\[
\maj(w) = \sum_{i \in \Des(w)} i.
\] We also need the following differential operators related to $\nabla$:
\[
\nabla^\beta = \nabla + \beta^2 \frac{\partial}{\partial \beta} 
\quad \text{and} \quad  
E = \sum_{i=1}^\infty x_i \frac{\partial}{\partial x_i}.
\]
We can now recall \cite[Theorem~A.1]{Pechenik.Speyer.Weigandt}, as reformulated in \cite[Remark~A.2]{Pechenik.Speyer.Weigandt}, an analogue of Proposition~\ref{prop:HPSW} for Grothendieck polynomials and our key tool in this subsection.

\begin{proposition}[{\cite[Theorem~A.1]{Pechenik.Speyer.Weigandt}}]\label{prop:PSW}
	For $w \in S_+$, we have
	\[
	\nabla^\beta \mathfrak{G}_w^{(\beta)} = \beta (\maj(w^{-1}) - \ell(w)) \mathfrak{G}_w^{(\beta)} + \sum_{s_kw < w} k \mathfrak{G}^{(\beta)}_{s_k w}
	\]
	and 
	\[
	(\maj(w^{-1}) + \nabla - E)\mathfrak{G}_w = \sum_{s_kw < w} k \mathfrak{G}_{s_k w}. 
	\]
\end{proposition}

Using essentially the same proof as for Proposition~\ref{prop:wlr}, but with \cite[Theorem~A.1]{Pechenik.Speyer.Weigandt} in place of \cite[Proposition~1.1]{Hamaker.Pechenik.Speyer.Weigandt}, we obtain the following analogue of Proposition~\ref{prop:wlr}, giving linear relations among $K$-theoretic Schubert structure coefficients.

\begin{proposition}\label{prop:grothendieck}
		Let $u,v,w \in S_+$. Then
\begin{align*}
\beta K_{u,v}^w(\beta) \bigg( &\maj(u^{-1}) + \maj(v^{-1}) - \maj(w^{-1}) - \ell(u) - \ell(v) + \ell(w) \bigg) \\
&+	\sum_{s_iu < u} i K_{s_iu,v}^w(\beta) + \sum_{s_jv < v} j K_{u,s_jv}^w(\beta) = \sum_{s_kw> w} k K_{u,v}^{s_k w}(\beta).
\end{align*} 
\end{proposition}
\begin{proof}
Write 
\[
\sum_p K_{u,v}^p(\beta) \mathfrak{G}^{(\beta)}_p=\mathfrak{G}^{(\beta)}_u \cdot \mathfrak{G}^{(\beta)}_v
\]
 and apply $\nabla^\beta$ to both sides, using the first part of Proposition~\ref{prop:PSW}.
Then on the left we have
\begin{align*}
\nabla^\beta \left( \sum_p K_{u,v}^p(\beta) \mathfrak{G}^{(\beta)}_p \right) &= \sum_p K_{u,v}^p(\beta) \nabla^\beta \mathfrak{G}^{(\beta)}_p \\ &= \sum_p K_{u,v}^p(\beta)
 \left( \beta \mathfrak{G}^{(\beta)}_p (\maj(p^{-1}) - \ell(p)) + \sum_{s_kp < p} k \mathfrak{G}^{(\beta)}_{s_k p} 
 \right),
 \end{align*}
 while on the right we have
 \begin{align*}
 	\nabla^\beta ( \mathfrak{G}^{(\beta)}_u \cdot \mathfrak{G}^{(\beta)}_v) &= \nabla^\beta (\mathfrak{G}^{(\beta)}_u) \mathfrak{G}^{(\beta)}_v +
 	 \mathfrak{G}^{(\beta)}_u (\nabla^\beta \mathfrak{G}^{(\beta)}_v) \\
 	 &= \left( \beta \mathfrak{G}^{(\beta)}_u (\maj(u^{-1}) - \ell(u)) + \sum_{s_iu < u} i \mathfrak{G}^{(\beta)}_{s_i u}
 	 \right) \mathfrak{G}^{(\beta)}_v \\ &+ \mathfrak{G}^{(\beta)}_u \left(\beta \mathfrak{G}^{(\beta)}_v (\maj(v^{-1}) - \ell(v)) + \sum_{s_j v < v} j \mathfrak{G}^{(\beta)}_{s_j v} \right).
 \end{align*}
 Now we can extract the coefficient of $\mathfrak{G}_w^{(\beta)}$ from both  of these obtain
\begin{align*}
 \sum_{s_k w > w} k K_{u,v}^{s_k w}(\beta) + \beta (&\maj(w^{-1}) - \ell(w)) K_{u,v}^w(\beta)  = \beta(\maj(v^{-1}) - \ell(v)) K_{u,v}^w(\beta)\\ & + \beta (\maj(u^{-1})  - \ell(u)) K_{u,v}^w(\beta) +\sum_{s_i u < u} i K_{s_i u, v}^w(\beta)+ \sum_{s_j v < v} j K_{u, s_jv}^w(\beta).
\end{align*}
The proposition follows by rearranging and collecting terms.
\end{proof}

Let $\mathfrak{G}_w(\mathbf 1)$ denote the specialization of the Grothendieck polynomial $\mathfrak{G}_w$ obtained by setting all variables equal to $1$.
It is well known to experts that $\mathfrak G_w( \mathbf 1)=1$ (see \cite{Smirnov.Tutubalina} for an explicit proof and \cite{Meszaros.Setiabrata.StDizier} for further discussion).  We present a new short proof using the second part of Proposition~\ref{prop:PSW}.

\begin{corollary}
	Given $w\in S_+$, $\mathfrak G_w( \mathbf 1)=1$.
\end{corollary}
\begin{proof}
	We proceed by induction on Coxeter length.  In the base case, $\mathfrak G_{\rm id}=1$ and there is nothing to show.  Now fix $w\in S_+$ with $\ell(w) \geq 1$ and assume the statement holds for all $v\in S_+$ with $\ell(v)<\ell(w)$.

	First note, for any $f\in \mathbb Z[x_1,x_2,\ldots]$, we have $(\nabla-E)(f)|_{\mathbf x=1}=0$.  Thus,
	\begin{align*}
	\maj(w^{-1})\mathfrak G_w(\mathbf 1)&=\sum_{s_kw<w}k\mathfrak G_{s_kw}(\mathbf 1) &\text{(by Proposition~\ref{prop:PSW})}\\
	&=\sum_{s_kw<w}k &\text{(by induction)}\\
	&=\maj(w^{-1}).
	\end{align*}
Since $\ell(w)\geq 1$, $\maj(w^{-1})\neq 0$ which implies $\mathfrak G_w(\mathbf 1)=1$. 
\end{proof}

\subsection{Stabilization}\label{sec:stabilization}

Recall the stabilization operation from Section~\ref{sec:proof11}.

The following is an analogue of Proposition~\ref{prop:wlr}; we drop the coefficients on the linear relations at the expense of adding one extra term.

\begin{proposition}\label{prop:stable}
	Let $u,v,w \in S_+$. Then
		\begin{equation}\label{eq:stablization}
	\sum_{s_iu < u}  c_{s_iu,v}^w + \sum_{s_jv < v}  c_{u,s_jv}^w = c_{1\times u,1\times v}^{s_1(1\times w)}+\sum_{s_kw> w}  c_{u,v}^{s_k w}.
	\end{equation}
\end{proposition}
\begin{proof}
	It follows easily from the pipe dream formula for Schubert polynomials (e.g., \cite{Fomin.Kirillov.96,Bergeron.Billey,Knutson.Miller}) that 
	\[\mathfrak S_{1\times w}(0,x_1,x_2,\dots)=\mathfrak S_w(x_1, x_2,\dots). 
	\] 
	   (Indeed, we will prove a stronger version of this statement as Lemma~\ref{lem:1}.) By Lemma~\ref{lem:easy}, $c_{u,v}^w=c_{1\times u,1\times v}^{1\times w}$.
	
	Apply Proposition~\ref{prop:wlr} to $1 \times u$, $1 \times v$, and $1\times w$.  Then we get
		\begin{equation}
		\mbox{ \tiny 
		$\displaystyle \sum_{s_{i+1}(1\times u) < 1\times u} (i+1) c_{s_{i+1}(1\times u),1\times v}^{1\times w} + \sum_{s_{j+1} (1\times v) < 1\times v} (j+1) c_{1\times u,s_{j+1} (1\times v)}^{1\times w} = \sum_{s_{k+1}(1\times w) > 1\times w} (k+1) c_{1\times u,1\times v}^{s_{k+1} (1\times w)}.$}
	\end{equation}
Thus,
		\begin{equation}\label{eq:stable}
	\sum_{s_{i} u < u} (i+1) c_{s_{i} u, v}^{ w} + \sum_{s_{j}v <  v} (j+1) c_{ u,s_{j}  v}^{w} = c_{1\times u,1\times v}^{s_1 (1\times w)}+ \sum_{s_{k} w>  w} (k+1) c_{ u, v}^{s_{k}  w}.
\end{equation}
Furthermore, from applying Proposition~\ref{prop:wlr} to $u$, $v$, and $w$, we have
\begin{equation}\label{eq:notstable}
	\sum_{s_iu < u} i c_{s_iu,v}^w + \sum_{s_jv < v} j c_{u,s_jv}^w = \sum_{s_kw> w} k c_{u,v}^{s_k w}. 
\end{equation}
Therefore, subtracting Equation~\eqref{eq:notstable} from Equation~\eqref{eq:stable} yields
\begin{align*}
	\sum_{s_iu < u}  c_{s_iu,v}^w + \sum_{s_jv < v}  c_{u,s_jv}^w & = c_{1\times u,1\times v}^{s_1(1\times w)}+\sum_{s_kw> w}  c_{u,v}^{s_k w}. \qedhere
\end{align*}
\end{proof}

\begin{remark}
	The stabilization argument in the proof of Proposition~\ref{prop:stable} is essentially equivalent to the use of the operator $\xi$ developed in \cite{Nenashev}.
\end{remark}

In many cases, one can see that the extra term on the right side of Proposition~\ref{prop:stable} is in fact $0$. In these situations, Proposition~\ref{prop:stable} becomes identical to Proposition~\ref{prop:wlr}, but \emph{with the coefficients dropped}, yielding two  linear relations among the same structure coefficients. In most cases, these relations are linearly independent.

One can also do analogous analysis in $K$-theory; we omit the details, since we will not use the $K$-theoretic analogue in what follows.

\subsection{Iterations of differential operators}\label{sec:iteration}

Iterating the application of $\nabla$ allows us to obtain additional linear relations among Schubert structure coefficients. These linear relations are somewhat more complicated to state, but the proof is analogous to that of Proposition~\ref{prop:wlr}.

Let $\mathfrak{S}_w(\mathbf 1)$ denote the specialization of the Schubert polynomial $\mathfrak{S}_w$ obtained by setting all variables equal to $1$. For $w$ with $\ell(w) = k$, a \newword{reduced word} for $w$ is a sequence $(a_1, a_2, \dots, a_k)$ of positive integers such that $w = s_{a_1} s_{a_2} \cdots s_{a_k}$. Let $R(w)$ denote the set of all reduced words for $w$. The Macdonald reduced word identity \cite[Eq.~(6.11)]{Macdonald:notes} is the following.

\begin{proposition}[{\cite[Eq.~(6.11)]{Macdonald:notes}}]\label{prop:Macdonald}
	Let $w \in S_+$ have $\ell(w) = k$. Then
	\[
	k! \mathfrak{S}_w(\mathbf 1) = \sum_{a \in R(w)} a_1 a_2 \cdots a_k.
	\]
\end{proposition}

We can use the Macdonald reduced word identity to obtain more linear relations.

\begin{proposition}
	\label{prop:morederivatives}
Let $u,v\in S_+$ and fix $1\leq k\leq \ell(u)+\ell(v)$.  Let $w\in S_+$ with \[
\ell(w)=\ell(u)+\ell(v)-k.
\]
 Then,
\begin{equation}\label{eq:morederivatives}
	\sum_{\substack{\hat{w}\geq_Lw \\ \ell(\hat{w})=\ell(w)+k}} \!\! c_{u,v}^{\hat{w}}\cdot \mathfrak S_{\hat{w}w^{-1}}(\mathbf 1)=\sum_{i=0}^{k} \sum_{\substack{\hat{u}\leq_L u \\ \ell(\hat{u})=\ell(u)-i}} \;\;\sum_{\substack{\hat{v} \leq_L v \\ \ell(\hat{v})=\ell(v)-(k-i)}} \!\!\!\! c_{\hat{u},\hat{v}}^{w}\cdot\mathfrak S_{u\hat{u}^{-1}}(\mathbf 1)  \cdot \mathfrak S_{v\hat{v}^{-1}}(\mathbf 1).
\end{equation}
\end{proposition}
\begin{proof}
	First note that
		\begin{equation*}
		\nabla^k(\mathfrak S_\pi)=\sum_{\substack{\hat{\pi}\leq_L \pi \\ \ell(\hat{\pi})=\ell(\pi)-k}}  \left(\sum_{a \in R(\pi \hat{\pi}^{-1})} a_1a_2 \cdots a_k \right)\cdot  \mathfrak S_{\hat{\pi}}
	\end{equation*}
	by Proposition~\ref{prop:HPSW}.
By Proposition~\ref{prop:Macdonald}, we can replace the inner summation to obtain
	\begin{equation}\label{eq:partial_macdonald}
		\nabla^k(\mathfrak S_\pi)=\sum_{\substack{\hat{\pi}\leq_L \pi \\ \ell(\hat{\pi})=\ell(\pi)-k}}  k!\mathfrak S_{\pi \hat{\pi}^{-1}}(\mathbf 1)\cdot  \mathfrak S_{\hat{\pi}}.
	\end{equation}
	
	Write 
	\[
	\sum_{\hat{p}}c_{u,v}^{\hat{p}} \mathfrak S_{\hat{p}}=\mathfrak S_u \cdot \mathfrak S_v.
	\]
	Applying $\nabla^k$ to both sides, we obtain
	\begin{equation*}
		\nabla^k\left(\sum_{\hat{p}} c_{u,v}^{\hat{p}} \mathfrak S_{\hat{p}} \right)=\nabla^k(\mathfrak S_u\cdot \mathfrak S_v)
		=\sum_{i=0}^{k} \binom{k}{i} \nabla^i(\mathfrak S_u)\nabla^{k-i}(\mathfrak S_v).
	\end{equation*}
Thus, Equation~\eqref{eq:partial_macdonald} yields
\begin{align*}
	\sum_{\hat{p}} c_{u,v}^{\hat{p}} & \sum_{\substack{p\leq_L \hat{p}  \\ \ell(p)=\ell(\hat{p})-k}}  k!\mathfrak S_{\hat{p}p^{-1}}(\mathbf 1) \mathfrak S_p
	\\
	&=\sum_{i=0}^{k} \binom{k}{i} \left(\sum_{\substack{\hat{u} \leq_Lu \\ \ell(\hat{u})=\ell(u)-i}}  \! i!\mathfrak S_{u\hat{u}^{-1}}(\mathbf 1) \mathfrak S_{\hat{u}}\right) \left (\sum_{\substack{\hat{v}\leq_L v \\ \ell(\hat{v})=\ell(v)-(k-i)}} \!\! \!(k-i)!\mathfrak S_{v\hat{v}^{-1}}(\mathbf 1) \mathfrak S_{\hat{v}} \right )\\
	&=\sum_{i=0}^{k} k! \sum_{\substack{\hat{u}\leq_L u \\ \ell(\hat{u})=\ell(u)-i}} \:\;\;  \sum_{\substack{\hat{v} \leq_L v \\ \ell(\hat{v})= \ell(v)-(k-i)}}  \mathfrak S_{u\hat{u}^{-1}}(\mathbf 1) \mathfrak S_{v\hat{v}^{-1}}(\mathbf 1)\sum_q c_{\hat{u},\hat{v}}^q \mathfrak S_q.
\end{align*}
Extract the coefficient of $\mathfrak S_w$ from each side to obtain
\begin{align*}
	 k! \!\!\!\!\sum_{\substack{\hat{w}\geq_Lw \\ \ell(\hat{w})=\ell(w)+k}} \!\!\! c_{u,v}^{\hat{w}} \mathfrak S_{\hat{w}w^{-1}}(\mathbf 1)= k! \sum_{i=0}^{k}  \sum_{\substack{\hat{u}\leq_Lu \\ \ell(\hat{u})=\ell(u)-i}} \;\;\;\sum_{\substack{\hat{v}\leq_Lv \\ \ell(\hat{v})=\ell(v)-(k-i)}} \!\!\!\! c_{\hat{u},\hat{v}}^{w}\mathfrak S_{u\hat{u}^{-1}}(\mathbf 1)   \mathfrak S_{v\hat{v}^{-1}}(\mathbf 1).
\end{align*}
The proposition follows by dividing out $k!$.
\end{proof}

Proposition~\ref{prop:wlr} may alternatively be proved as a corollary to Proposition~\ref{prop:morederivatives} by setting $k=1$. On the other extreme, we have the following corollary. 
Define
\[
\delta(w,i) = \begin{cases}
	1, & \text{if $i \in \Des(w)$;} \\ 
	0, & \text{if $i \notin \Des(w)$.}
\end{cases}
\]

\begin{corollary}\label{cor:Kronecker}
Let $u,v\in S_+$ and fix $i \in \mathbb{Z}_{>0}$.  Then,
\begin{equation}\label{eq:Kronecker}
\sum_{p: i \in \Des(p)} c_{u,v}^{p}\mathfrak S_{ps_i}(\mathbf 1)=\delta(v,i)\mathfrak S_u(\mathbf 1)\mathfrak S_{vs_i}(\mathbf 1)+\delta(u,i)\mathfrak S_{us_i}(\mathbf 1)\mathfrak S_{v}(\mathbf 1).
\end{equation}
\end{corollary}
\begin{proof}
	 In Proposition~\ref{prop:morederivatives}, take $w=s_i$ and  $k = \ell(u) + \ell(v) - 1$. The left side of Equation~\eqref{eq:morederivatives} becomes the left side of Equation~\eqref{eq:Kronecker}. Most of the terms on the right side of Equation~\eqref{eq:morederivatives} vanish by degree considerations.  This leaves only the terms on the right side of Equation~\eqref{eq:Kronecker}.
\end{proof}

The following was observed as \cite[Lemma~1.1]{Knutson:cycling}, where the phenomenon was referred to as \emph{dc-triviality}. We obtain another easy proof of this fact.

\begin{corollary}[{\cite[Lemma~1.1]{Knutson:cycling}}]
Let $u,v\in S_+$ and suppose $i \notin \Des(u) \cup \Des(v)$.
Then, $c_{u,v}^p = 0$ for all $p \in S_+$ with $i \in \Des(p)$.
\end{corollary}
\begin{proof}
	In Corollary~\ref{cor:Kronecker}, both terms on the right side vanish under these hypotheses. The left side is a sum of nonnegative integers, so all terms on the left side also vanish. Since each specialization $\mathfrak{S}_p(\mathbf 1)$ is strictly positive, all the relevant $c_{u,v}^p$ equal zero.
\end{proof}

\section{Discussion and proof of \texorpdfstring{Theorem~\ref{thm:conj2}}{Theorem 1.2}}\label{sec:proof12}
For a permutation $w \in S_+$ and positive integer $c$, we define $\tau_c^{-1}(w)$ to be $s_{c} \cdots s_2s_1(1 \times w)$. 
For a permutation $w \in S_+$, we define the \newword{truncation} $\tau(w)$ to be the permutation such that $\tau_{c_1(w)}^{-1}(\tau(w)) = w$.
Note that $\tau(w)$ is the unique element of $S_+$ such that 
\[
c_i(\tau(w)) = 
\begin{cases}
	c_{i+1}(w), & \text{if } i>0; \\
	0, & \text{otherwise}.
\end{cases}
\]

The following lemmas are closely related to ideas of \cite{Bergeron.Sottile,Lenart.Robinson.Sottile} and are known to experts; we include (sketches of) proofs for completeness.

\begin{lemma}\label{lem:1}
	For any $w \in S_n$, we have
	\[
	\mathfrak{S}_w(1, x_1, x_2, \dots, x_{n-1}) = \mathfrak{S}_{\tau(w)}(x_1, x_2, \dots, x_n) + \text{{\rm lower degree terms}}.
	\]
\end{lemma}
\begin{proof}
	The Schubert polynomial $\mathfrak{S}_w$ can be written as a generating function for diagrams $P$ called \emph{pipe dreams} (cf.\ \cite{Bergeron.Billey, Knutson.Miller}), where each $\plus$ in row $i$ of $P$ contributes the variable $x_i$ to the weight of $P$. Under the specialization of the lemma, the highest-degree terms of  $\mathfrak{S}_w(1, x_1, x_2, \dots, x_{n-1})$ come from pipe dreams with as few $\plus$'s as possible in the first row.
	
	The \emph{ladder moves} of \cite{Bergeron.Billey} describe a recursive algorithm to generate all pipe dreams for $w$. From this algorithm, it is straightforward that the pipe dreams for $w$ with a minimum number of $\plus$'s in the first row are identical to the pipe dreams for $\tau(w)$ after deleting their first row and shifting up. 
\end{proof}

\begin{lemma}\label{lem:2}
	Let $u,v,w \in S_n$ such that $c_1(w) = c_1(u) + c_1(v)$. Then we have
	\[
	c_{u,v}^w = c_{\tau(u), \tau(v)}^{\tau(w)}.
	\]
\end{lemma}
\begin{proof}
	Write 
	\[
	\mathfrak{S}_u \mathfrak{S}_v = \sum_a c_{u,v}^a \mathfrak{S}_a.
	\] Choose $m$ sufficiently large so that $a \in S_m$ for all $a$ such that $c_{u,v}^a \neq 0$.
	We may specialize all of the variables in this equation to obtain
	\[
	\mathfrak{S}_u(1, x_1, \dots, x_{m-1}) \mathfrak{S}_v(1, x_1, \dots, x_{m-1}) = \sum_a c_{u,v}^a \mathfrak{S}_a(1, x_1, \dots, x_{m-1}).
	\]
	
	Now, by Lemma~\ref{lem:1} applied to all of these Schubert polynomials, we find that
\begin{equation}\label{eq:trick}
	(\mathfrak{S}_{\tau(u)} +f)(\mathfrak{S}_{\tau(v)}+g) = \sum_a c_{u,v}^a (\mathfrak{S}_{\tau(a)} + h_a),
\end{equation}
	where $\deg f < \ell(\tau(u))$, $\deg g < \ell(\tau(v))$ and $\deg h_a < \ell(\tau(a))$ for each $a$. Now observe that $\ell(\tau(a)) = \ell(\tau(u)) + \ell(\tau(v))$ if and only if we have $c_1(a) = c_1(u) + c_1(v)$.
	Therefore, by extracting the top-degree terms on both sides of Equation~\eqref{eq:trick}, we obtain
\begin{equation}\label{eq:first_expand}
	\mathfrak{S}_{\tau(u)} \mathfrak{S}_{\tau(v)} = \sum_{c_1(b) = c_1(u) + c_1(v)} c_{u,v}^b \mathfrak{S}_{\tau(b)}.
\end{equation}
	On the other hand, by definition,
\begin{equation}\label{eq:second_expand}
	\mathfrak{S}_{\tau(u)} \mathfrak{S}_{\tau(v)} = \sum_d c_{\tau(u), \tau(v)}^d \mathfrak{S}_d.
\end{equation}
Now, observe that if $\tau(b_1) = \tau(b_2)$ and $c_1(b_1) = c_1(b_2)$, then necessarily $b_1 = b_2$.
	Therefore, comparing Equations~\eqref{eq:first_expand} and~\eqref{eq:second_expand} yields
	\[
	c_{\tau(u), \tau(v)}^{\tau(w)} = c_{u,v}^w,
	\]
	as desired.
\end{proof}

Note that Lemma~\ref{lem:easy} is a special case of Lemma~\ref{lem:2}. With these lemmas in hand, we are now prepared to finish the proof of Theorem~\ref{thm:conj2}.

\begin{proof}[Proof of Theorem~\ref{thm:conj2}]
Recall that $u,v,w \in S_+$ are permutations, where
	 $u$ is $p$-inverse Grassmannian and $v$ is $q$-inverse Grassmannian. Note that this implies that $p,q > 0$. Let $n = p+q$. We consider a few cases in turn.

We can assume that $\ell(w) = \ell(u) + \ell(v) - 1$ because otherwise we certainly have $c_{s_pu, v}^w = 0$.

	\medskip
		\noindent
	{\sf (Case 1: $p\neq q$):} We first consider a technical special case, which we will then be able to extend.
	
	\medskip
		\noindent
	{\sf (Case 1.1: 
	$c_{1 \times u, 1 \times v}^{s_1(1 \times w)} = 0$):} In this case, we have by Proposition~\ref{prop:stable} that
\begin{equation}\label{eq:zeroed}
	c_{s_p u, v}^w + c_{u, s_q v}^w = \sum_{s_k w > w} c_{u,v}^{s_k w}
\end{equation}
and from Proposition~\ref{prop:wlr} that
\begin{equation}\label{eq:zeroed2}
	pc_{s_p u, v}^w + qc_{u, s_q v}^w = \sum_{s_k w > w} kc_{u,v}^{s_k w}.
\end{equation}
Thus, multiplying Equation~\eqref{eq:zeroed} by $q$ and subtracting from Equation~\eqref{eq:zeroed2}, we find that
\begin{equation}\label{eq:zeroed3}
	(p-q)c_{s_p u, v}^w  = \sum_{s_k w > w} (k-q)c_{u,v}^{s_k w}.
\end{equation} 

\medskip
	\noindent
	{\sf (Case 1.1.1: $T_w \cdot \gamma_{u,v}$ is not almost rainbow):} 	
By Theorem~\ref{thm:conj1}, $c_{u,v}^{s_k w}=0$ for all $k$ such that $s_k w > w$. Therefore, the right side of Equation~\eqref{eq:zeroed3} is $0$. Since $p\neq q$, this implies that 
	$c_{s_p u, v}^w = 0$. 

\medskip
	\noindent
	{\sf (Case 1.1.2: $T_w \cdot \gamma_{u,v}$ is almost rainbow):} 
Write $\delta = T_w \cdot \gamma_{u,v}$. We now break into cases according to what sort of almost rainbow clan $\delta$ is. Observe that $\Omega_{p,q}$ has a nonzero number of signed unmatched numbers. If $p < q$, these signs are all $\minus$ and appear in positions $p+1, \dots, q$; if $p>q$, these signs are all $\plus$ and appear in positions $q+1, \dots, p$.

\medskip
	\noindent
	{\sf (Case 1.1.2.1: $T_{s_q} \cdot \delta = \Omega_{p,q}$):} 
	We observe that $T_{s_q}$ must act on $\delta$ by moving a sign inside an arc (as in the $T_1$ or $T_6$ arrow of Figure~\ref{fig:heck_action}).
	Therefore, we have $T_{s_r} \cdot \delta = \delta$ for all $r \neq q$. 
	So, in this case, Equation~\eqref{eq:zeroed3} simplifies to 
	\[
	(p-q)c_{s_p u, v}^w  = (q-q)c_{u,v}^{s_q w} = 0.
	\] Since $p \neq q$, we therefore have 
	 $c_{s_p u, v}^w =0$.
		
	\medskip
	\noindent
	{\sf (Case 1.1.2.2: $T_{s_p} \cdot \delta = \Omega_{p,q}$):} 
	We observe that $T_{s_p}$ must again act on $\delta$ by moving a sign inside an arc, as in the previous case.
	Therefore, we have $T_{s_r} \cdot \delta = \delta$ for all $r \neq p$. 
	So, in this case, Equation~\eqref{eq:zeroed3} simplifies to 
	\[
	(p-q)c_{s_p u, v}^w  = (p-q)c_{u,v}^{s_p w} .
	\] Since $p \neq q$, we therefore have 
	 $c_{s_p u, v}^w  = c_{u,v}^{s_p w} =1$, where the last equality is by Theorem~\ref{thm:conj1}.
	 
	\medskip
	\noindent
	{\sf (Case 1.1.2.3: $T_{s_r} \cdot \delta = \Omega_{p,q}$ for some $r \notin \{p,q \}$):}  In this case, $T_{s_r}$ must act on $\delta$ by uncrossing a pair of adjacent arcs (as in the $T_2$ arrow of Figure~\ref{fig:heck_action}). 
	Recall that $n-r$ and $n-r+1$ are the labels on the other ends of the crossing arcs from $r, r+1$.
	
	\medskip
	\noindent
	{\sf (Case 1.1.2.3.1: $n-r > 0$):} 	
In this case, we also have $T_{s_{n-r}} \cdot \delta = \Omega_{p,q}$. Observe that $r \neq n-r$. Moreover, we have $T_{s_{t}} \cdot \delta =  \delta$ for all $t \notin \{ r, n-r\}$. By Theorem~\ref{thm:conj1},
	\[
	c_{u,v}^{s_r w} = 1 = c_{u,v}^{s_{n-r} w}.
	\]
	Now, Equation~\eqref{eq:zeroed3} simplifies to \[
	(p-q)c_{s_p u, v}^w  =  (r-q)c_{u,v}^{s_r w} + (n-r-q)c_{u,v}^{s_{n-r} w} = (r-q) + (n-r-q) = n-2q = p-q.
	\]
	Thus, $c_{s_p u, v}^w=1$.
	
		\medskip
	\noindent
	{\sf (Case 1.1.2.3.2: $n-r \leq 0$):} 	
	Define $\tilde{u} = 1^{r-n+1} \times u$, $\tilde{v} = 1^{r-n+1} \times v$, and $\tilde{w} = 1^{r-n+1} \times w$. Then $\tilde{u}$ is $\tilde{p}$-inverse Grassmannian, where $\tilde{p} = p +r - n + 1$.
	Also let $\tilde{\gamma} = \gamma_{\tilde{u}, \tilde{v}}$ and notice that $\tilde{\gamma}$ is a horizontal shift of the backstable clan $\gamma_{u,v}$. Therefore, $T_{\tilde{w}} \cdot \tilde{\gamma}$ is a horizontal shift of $T_w \cdot \gamma_{u,v} = \delta$. In particular, $T_{\tilde{w}} \cdot \tilde{\gamma}$ is almost rainbow with a pair of crossing arcs. Observe that $s_{\tilde{p}} \tilde{u} = 1^{r-n+1} \times s_p u$. By the previous {\sf Case 1.1.2.3.1}, 
	\[
	c_{s_{\tilde{p}} \tilde{u}, \tilde{v}}^{\tilde{w}} = c_{1^{r-n+1} \times s_p u, 1^{r-n+1} \times v}^{1^{r-n+1} \times w} = 1.
	\] 
	Now, Lemma~\ref{lem:easy} gives that 
	\[
	c_{1^{r-n+1} \times s_p u, 1^{r-n+1} \times v}^{1^{r-n+1} \times w} = c_{s_p u, v}^w, 
	\]
	so $c_{s_p u, v}^w = 1$, as desired.

	\medskip
	\noindent
	{\sf (Case 1.2: $c_{1 \times u, 1 \times v}^{s_1(1 \times w)}\neq 0$):} 
	Define $\tilde{u} = 1 \times u$, $\tilde{v} = 1 \times v$, and $\tilde{w} = 1 \times w$. Then $\tilde{u}$ is $\tilde{p}$-inverse Grassmannian and $\tilde{v}$ is $\tilde{q}$-inverse Grassmannian, where $\tilde{p} = p + 1$ and $\tilde{q} = q+1$.
	Also let $\tilde{\gamma} = \gamma_{\tilde{u}, \tilde{v}}$ and notice that $\tilde{\gamma}$ is a horizontal shift of $\gamma_{u,v}$. 
	
	By Proposition~\ref{prop:stable}, we have 
	\begin{equation}\label{eq:unshifted}
	c_{s_p u, v}^w + c_{u, s_q v}^w = c_{\tilde{u}, \tilde{v}}^{s_1\tilde{w}} + \sum_{s_k w > w} c_{u,v}^{s_k w}
	\end{equation}
	and 
	\begin{equation}\label{eq:shifted}
	c_{s_{\tilde{p}} \tilde{u}, \tilde{v}}^{\tilde{w}} + c_{\tilde{u}, s_{\tilde{q}} \tilde{v}}^{\tilde{w}} = c_{1 \times \tilde{u}, 1\times  \tilde{v}}^{s_1(1 \times \tilde{w})} + \sum_{s_h \tilde{w} > \tilde{w}} c_{\tilde{u},\tilde{v}}^{s_h \tilde{w}}.
	\end{equation}

By Lemma~\ref{lem:easy}, we have 
\[
c_{s_p u, v}^w = c_{s_{\tilde{p}} \tilde{u}, \tilde{v}}^{\tilde{w}}, \quad c_{u, s_q v}^w = c_{\tilde{u}, s_{\tilde{q}} \tilde{v}}^{\tilde{w}}, \quad \text{and } c_{\tilde{u}, \tilde{v}}^{s_1\tilde{w}} + \sum_{s_k w > w} c_{u,v}^{s_k w} = \sum_{s_h \tilde{w} > \tilde{w}} c_{\tilde{u},\tilde{v}}^{s_h \tilde{w}}.
\] Thus, subtracting Equation~\eqref{eq:unshifted} from Equation~\eqref{eq:shifted} yields that
$c_{1 \times \tilde{u}, 1 \times \tilde{v}}^{s_1(1 \times \tilde{w})} = 0$.

		Now, $T_{\tilde{w}} \cdot \tilde{\gamma}$ is a horizontal shift of $T_w \cdot \gamma_{u,v}$. Observe that $s_{\tilde{p}} \tilde{u} = 1 \times s_p u$. 
	Since Lemma~\ref{lem:easy} gives that 
$c_{s_p u, v}^w = c_{s_{\tilde{p}} \tilde{u}, \tilde{v}}^{\tilde{w}}$ and the latter coefficient falls under {\sf Case 1.1}, we are done.

	\medskip
	\noindent
	{\sf (Case 2: $p = q$):}
	We establish this case by reduction to {\sf Case 1} via stabilization. Choose an interval $[i,j]$ on which $\gamma_{u,v}$ is supported.
	
	\medskip
	\noindent
	{\sf (Case 2.1: $i\geq 1$):}
	If $i >1$, expand the interval $[i,j]$ until $i=1$.
	
	Define $\tilde{u} = 1 \times u$, $\tilde{v} = \tau_p^{-1}(v)$, and $\tilde{w} = \tau_p^{-1}(w)$. 
	Then $\tilde{u}$ is $\tilde{p}$-inverse Grassmannian, where $\tilde{p} = p+1$. On the other hand, $\tilde{v}$ is $p$-inverse Grassmannian.
	
	By Lemma~\ref{lem:2}, we have 
	\[
	c_{1 \times (s_p u), \tilde{v}}^{\tilde{w}} = c_{\tau(1 \times (s_p u)), \tau(\tilde{v})}^{\tau(\tilde{w})} = c_{s_p u, v}^w.
	\]
	But also $1 \times (s_p u) = s_{\tilde{p}} \tilde{u}$, so 
	\begin{equation}\label{eq:key_equivalence}
	c_{s_{\tilde{p}} \tilde{u}, \tilde{v}}^{\tilde{w}} = c_{s_p u, v}^w.
	\end{equation}
Since $\tilde{p} \neq p$, the coefficient $c_{s_{\tilde{p}} \tilde{u}, \tilde{v}}^{\tilde{w}}$ falls under {\sf Case 1}. Note also that $\ell(\tilde{w}) = \ell(s_{\tilde{p}} \tilde{u}) + \ell(\tilde{v})$.

Let $\tilde{\gamma} = \gamma_{\tilde{u}, \tilde{v}}$. Notice that $\tilde{\gamma}$ is supported on $[1,j+1]$ and
 is obtained from $\gamma_{u,v}$ by placing a $\plus$ in position $1$ and shifting the rest of $\gamma_{u,v}$ horizontally one space to the right. That is,
 \[
 \tilde{\gamma}(z) = 
 \begin{cases}
 	\plus, & \text{if } z = 1;\\
 	\gamma_{u,v}(z-1) & \text{if } z > 1.
 \end{cases}
\]
See Example~\ref{ex:plus_shifty} for an illustration of this construction.

Let $\delta = T_{w} \cdot \gamma_{u,v}$ and let $\hat{\delta} = T_{1 \times w} \cdot \tilde{\gamma}$. Notice that $\hat{\delta}$ is similarly obtained from $\delta$ by placing a $\plus$ in position $1$ and shifting the rest of $\delta$ horizontally one space to the right.

Now notice that $\tilde{w} =  s_p s_{p-1} \cdots s_1(1\times w)$, so 
\[T_{\tilde{w}} \cdot \tilde{\gamma} = T_p T_{p-1} \cdots T_1 T_{1 \times w} \cdot \tilde{\gamma} = T_p T_{p-1} \cdots T_1 \cdot \hat{\delta}.
\]

	\medskip
	\noindent
	{\sf (Case 2.1.1: $\delta$ is not almost rainbow):} Let $h$ be the least positive integer such that there is a permutation $\theta$ with $\ell(\theta) = h$ and $T_\theta \cdot \delta = \Omega_{p,p}$. Then, it is easy to see that any permutation $\sigma$ with $T_\sigma \cdot \hat{\delta} = \Omega_{\tilde{p},p}$ must have $\ell(\sigma) \geq h + p$. In particular, 
	$T_{\tilde{w}} \cdot  \tilde{\gamma} =  T_p T_{p-1} T_{p-2} \cdots T_1 \cdot \hat{\delta}$ is not almost rainbow. Therefore, by {\sf Case 1}, we have $c_{s_{\tilde{p}} \tilde{u}, \tilde{v}}^{\tilde{w}} = 0$. Therefore, Equation~\eqref{eq:key_equivalence} gives that $c_{s_p u, v}^w = 0$, as desired.

	\medskip
	\noindent
	{\sf (Case 2.1.2: $\delta$ is almost rainbow):} We break into cases according to what type of almost rainbow clan $\delta$ is.
	
	\medskip
	\noindent
	{\sf (Case 2.1.2.1: $T_{s_p} \cdot \delta = \Omega_{p,p}$):} 
	In this case, $T_{s_p}$ must act on $\delta$ by joining a $\plus$ and a $\minus$ into an arc (as in the $T_4$ arrow of Figure~\ref{fig:heck_action}). The action of $T_{p-1} T_{p-2} \cdots T_1$ on $\hat{\delta}$ is then to move another $\plus$ from position $1$ past $p-1$ initial nodes to land in position $p$.

	\medskip
	\noindent
	{\sf (Case 2.1.2.1.1 $\delta(p) = \plus$):}
	Here, $T_{p-1} T_{p-2} \cdots T_1 \cdot \hat{\delta}(p) = \plus$ and $T_{p-1} T_{p-2} \cdots T_1 \cdot \hat{\delta}(p+1) = \plus$. Hence, $T_p T_{p-1} T_{p-2} \cdots T_1 \cdot \hat{\delta} = T_{p-1} T_{p-2} \cdots T_1 \cdot \hat{\delta}$ and, in particular, $T_p T_{p-1} T_{p-2} \cdots T_1 \cdot \hat{\delta}$ is not an almost rainbow clan. Thus, by {\sf Case 1}, we then have $c_{s_{\tilde{p}} \tilde{u}, \tilde{v}}^{\tilde{w}} = 0$. Therefore, Equation~\eqref{eq:key_equivalence} gives that $c_{s_p u, v}^w = 0$, as desired.		
	
	\medskip
	\noindent
	{\sf (Case 2.1.2.1.2 $\delta(p) = \minus$):}
	Here, $T_{p-1} T_{p-2} \cdots T_1 \cdot \hat{\delta}(p) = \plus$ and $T_{p-1} T_{p-2} \cdots T_1 \cdot \hat{\delta}(p+1) = \minus$. Hence, $T_p$ acts on $T_{p-1} T_{p-2} \cdots T_1 \cdot \hat{\delta}$ by joining these $\plus$ and $\minus$ into an arc. So $T_p T_{p-1} T_{p-2} \cdots T_1 \cdot \hat{\delta}$ is an almost rainbow clan in $\Psi_{\tilde{p}, p}$.
	Thus, by {\sf Case 1}, we then have $c_{s_{\tilde{p}} \tilde{u}, \tilde{v}}^{\tilde{w}} = 1$. Therefore, Equation~\eqref{eq:key_equivalence} gives that $c_{s_p u, v}^w = 1$, as desired.		
	
	\medskip
	\noindent
	{\sf (Case 2.1.2.2: $T_{s_r} \cdot \delta = \Omega_{p,p}$ for some $r \neq p$):}
	In this case, $T_{s_r}$ must act on $\delta$ by uncrossing a pair of adjacent arcs. 
	Recall that $n-r$ and $n-r+1$ are the labels on the other ends of the crossing arcs from $r, r+1$.
	
	The action of $T_p T_{p-1} \cdots T_1$ on $\hat{\delta}$ is then to move the $\plus$ from position $1$ past $p$ initial nodes to land in position $p+1$. Thus, $T_{\tilde{w}} \cdot \tilde{\gamma} = T_p T_{p-1} \cdots T_1 \cdot \hat{\delta}$ is an almost rainbow clan in $\Psi_{\tilde{p},p}$. By {\sf Case 1}, we then have $c_{s_{\tilde{p}} \tilde{u}, \tilde{v}}^{\tilde{w}} = 1$. Therefore, Equation~\eqref{eq:key_equivalence} gives that $c_{s_p u, v}^w = 1$, as desired.
	
	\medskip
	\noindent
	{\sf (Case 2.2: $i < 1$):}
		Define $\tilde{u} = 1^{1-i} \times u$, $\tilde{v} = 1^{1-i} \times v$, and $\tilde{w} = 1^{1-i} \times w$. Then $\tilde{u}$ and $\tilde{v}$ are $\tilde{p}$-inverse Grassmannian, where $\tilde{p} = p -i + 1$.
		
	Also let $\tilde{\gamma} = \gamma_{\tilde{u}, \tilde{v}}$ and notice that $\tilde{\gamma}$ is a horizontal shift of $\gamma_{u,v}$. 
	Therefore, $T_{\tilde{w}} \cdot \tilde{\gamma}$ is a horizontal shift of $T_w \cdot \gamma_{u,v} = \delta$. 
	
	 Observe that $s_{\tilde{p}} \tilde{u} = 1^{1-i} \times s_p u$. 
	 Since Lemma~\ref{lem:easy} gives that 
$c_{s_p u, v}^w = c_{s_{\tilde{p}} \tilde{u}, \tilde{v}}^{\tilde{w}}$ and the latter coefficient falls under {\sf Case 2.1}, we are done.
\end{proof}
	
	\begin{example}\label{ex:plus_shifty}
	 We illustrate part of the construction from {\sf Case 2.1}. Let $u = 51236748$ and $v = 12354678$. Here, $p=q=4$. The clan $\gamma_{u,v}$ is 
	  
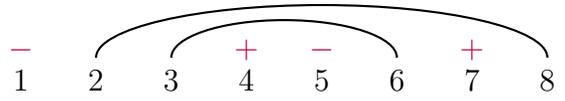
\[
\begin{tikzpicture}
				\foreach \i in {1,...,8}
	{
	\node (\i) at (\i,0) {\i };	
	}
    \draw[thick] (6,0.3) arc
    [
        start angle=0,
        end angle=180,
        x radius=1.5cm,
        y radius =0.5cm
    ] ;
    \draw[thick] (8,0.3) arc
    [
        start angle=0,
        end angle=180,
        x radius=3.0cm,
        y radius =0.7cm
    ] ;
	\node (1lab) at (1,0.4) {$\minus$};
	\node (4lab) at (4,0.4) {$\plus$};
	\node (5lab) at (5,0.4) {$\minus$};
	\node (7lab) at (7,0.4) {$\plus$};
		\end{tikzpicture}.
		\]
		
		Define $\tilde{u} = 1 \times u$ and $\tilde{v} = \tau_p^{-1}(v)$. In this case, $\tilde{u} = 162347859$ and $\tilde{v} = 512364789$.
	Note that $\tilde{u}$ is $5$-inverse Grassmannian, while $\tilde{v}$ is $4$-inverse Grassmannian.

		Let $\tilde{\gamma} = \gamma_{\tilde{u}, \tilde{v}}$, which looks like
			  
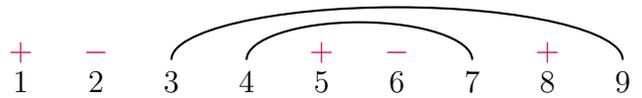
\[
\begin{tikzpicture}
				\foreach \i in {1,...,9}
	{
	\node (\i) at (\i,0) {\i };	
	}
    \draw[thick] (7,0.3) arc
    [
        start angle=0,
        end angle=180,
        x radius=1.5cm,
        y radius =0.5cm
    ] ;
    \draw[thick] (9,0.3) arc
    [
        start angle=0,
        end angle=180,
        x radius=3.0cm,
        y radius =0.7cm
    ] ;
    \node (0lab) at (1,0.4) {$\plus$};
	\node (1lab) at (2,0.4) {$\minus$};
	\node (4lab) at (5,0.4) {$\plus$};
	\node (5lab) at (6,0.4) {$\minus$};
	\node (7lab) at (8,0.4) {$\plus$};
		\end{tikzpicture}.
		\]
		Notice that $\tilde{\gamma}$
 is obtained from $\gamma_{u,v}$ by placing a $\plus$ in position $1$ and shifting the rest of $\gamma_{u,v}$ to the right, as described in {\sf Case 2.1}. 
	\end{example}

	\begin{example}\label{ex:conj2}
		 Let $u = 3142$. Note that $u$ is $2$-inverse Grassmannian and that $s_2 u = 2143$. We use Theorem~\ref{thm:conj2} to compute the Schubert structure coefficients $c_{2143,3142}^w$ for all $w \in S_+$. 
		 We have that $\gamma_{u,u}$ looks like 
		 	  
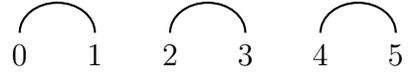
\[
\begin{tikzpicture}
				\foreach \i in {0,...,5}
	{
	\node (\i) at (\i,0) {\i };	
	}
    \draw[thick] (1,0.3) arc
    [
        start angle=0,
        end angle=180,
        x radius=0.5cm,
        y radius =0.4cm
    ] ;
    \draw[thick] (3,0.3) arc
    [
        start angle=0,
        end angle=180,
        x radius=0.5cm,
        y radius =0.4cm
    ] ;
        \draw[thick] (5,0.3) arc
    [
        start angle=0,
        end angle=180,
        x radius=0.5cm,
        y radius =0.4cm
    ] ;
		\end{tikzpicture}.
		\]
		We consider all nontrivial actions of $0$-Hecke generators $T_i$ on $\gamma_{u,u}$, until reaching an almost rainbow clan:
				\begin{center}
	\scalebox{0.72}
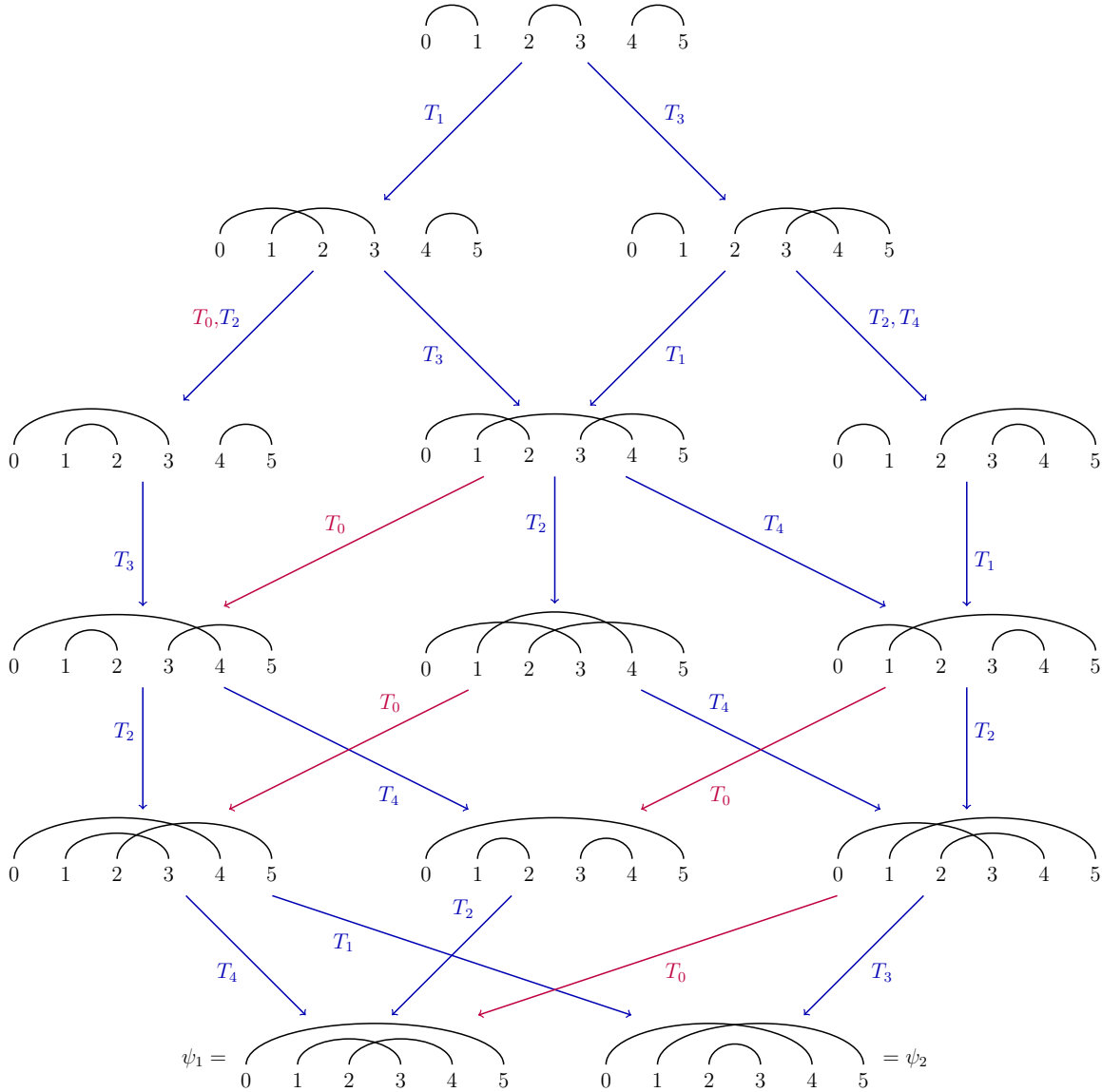
{
\begin{tikzpicture}
	\node (N1) at (0,0) {
		\begin{tikzpicture}
				\foreach \i in {0,...,5}
	{
	\node (\i) at (\i,0) {\i };	
	}
    \draw[thick] (1,0.3) arc
    [
        start angle=0,
        end angle=180,
        x radius=0.5cm,
        y radius =0.4cm
    ] ;
    \draw[thick] (3,0.3) arc
    [
        start angle=0,
        end angle=180,
        x radius=0.5cm,
        y radius =0.4cm
    ] ;
        \draw[thick] (5,0.3) arc
    [
        start angle=0,
        end angle=180,
        x radius=0.5cm,
        y radius =0.4cm
    ] ;
		\end{tikzpicture}};
    \node (N2) at (-4,-4) {
			\begin{tikzpicture}
				\foreach \i in {0,...,5}
	{
	\node (\i) at (\i,0) {\i };	
	}
    \draw[thick] (2,0.3) arc
    [
        start angle=0,
        end angle=180,
        x radius=1.0cm,
        y radius =0.5cm
    ] ;
    \draw[thick] (3,0.3) arc
    [
        start angle=0,
        end angle=180,
        x radius=1.0cm,
        y radius =0.5cm
    ] ;
        \draw[thick] (5,0.3) arc
    [
        start angle=0,
        end angle=180,
        x radius=0.5cm,
        y radius =0.4cm
    ] ;
		\end{tikzpicture}};
	\node (N3) at (4,-4) {
			\begin{tikzpicture}
				\foreach \i in {0,...,5}
	{
	\node (\i) at (\i,0) {\i };	
	}
    \draw[thick] (1,0.3) arc
    [
        start angle=0,
        end angle=180,
        x radius=0.5cm,
        y radius =0.4cm
    ] ;
    \draw[thick] (4,0.3) arc
    [
        start angle=0,
        end angle=180,
        x radius=1.0cm,
        y radius =0.5cm
    ] ;
        \draw[thick] (5,0.3) arc
    [
        start angle=0,
        end angle=180,
        x radius=1.0cm,
        y radius =0.5cm
    ] ;
		\end{tikzpicture}};
	\node (N4) at (-8,-8) {
			\begin{tikzpicture}
				\foreach \i in {0,...,5}
	{
	\node (\i) at (\i,0) {\i };	
	}
    \draw[thick] (2,0.3) arc
    [
        start angle=0,
        end angle=180,
        x radius=0.5cm,
        y radius =0.4cm
    ] ;
    \draw[thick] (3,0.3) arc
    [
        start angle=0,
        end angle=180,
        x radius=1.5cm,
        y radius =0.7cm
    ] ;
        \draw[thick] (5,0.3) arc
    [
        start angle=0,
        end angle=180,
        x radius=0.5cm,
        y radius =0.4cm
    ] ;
		\end{tikzpicture}};
		\node (N5) at (0,-8) {
			\begin{tikzpicture}
				\foreach \i in {0,...,5}
	{
	\node (\i) at (\i,0) {\i };	
	}
    \draw[thick] (2,0.3) arc
    [
        start angle=0,
        end angle=180,
        x radius=1.0cm,
        y radius =0.5cm
    ] ;
    \draw[thick] (4,0.3) arc
    [
        start angle=0,
        end angle=180,
        x radius=1.5cm,
        y radius =0.5cm
    ] ;
        \draw[thick] (5,0.3) arc
    [
        start angle=0,
        end angle=180,
        x radius=1.0cm,
        y radius =0.5cm
    ] ;
		\end{tikzpicture}};
		\node (6) at (8,-8) {
			\begin{tikzpicture}
				\foreach \i in {0,...,5}
	{
	\node (\i) at (\i,0) {\i };	
	}
    \draw[thick] (1,0.3) arc
    [
        start angle=0,
        end angle=180,
        x radius=0.5cm,
        y radius =0.4cm
    ] ;
    \draw[thick] (4,0.3) arc
    [
        start angle=0,
        end angle=180,
        x radius=0.5cm,
        y radius =0.4cm
    ] ;
        \draw[thick] (5,0.3) arc
    [
        start angle=0,
        end angle=180,
        x radius=1.5cm,
        y radius =0.7cm
    ] ;
		\end{tikzpicture}};
		\node (7) at (-8,-12) {
			\begin{tikzpicture}
				\foreach \i in {0,...,5}
	{
	\node (\i) at (\i,0) {\i };	
	}
    \draw[thick] (2,0.3) arc
    [
        start angle=0,
        end angle=180,
        x radius=0.5cm,
        y radius =0.4cm
    ] ;
    \draw[thick] (4,0.3) arc
    [
        start angle=0,
        end angle=180,
        x radius=2.0cm,
        y radius =0.7cm
    ] ;
        \draw[thick] (5,0.3) arc
    [
        start angle=0,
        end angle=180,
        x radius=1.0cm,
        y radius =0.5cm
    ] ;
		\end{tikzpicture}};
		\node (7andthreequarters) at (0,-12) {
			\begin{tikzpicture}
				\foreach \i in {0,...,5}
	{
	\node (\i) at (\i,0) {\i };	
	}
    \draw[thick] (3,0.3) arc
    [
        start angle=0,
        end angle=180,
        x radius=1.5cm,
        y radius =0.6cm
    ] ;
    \draw[thick] (4,0.3) arc
    [
        start angle=0,
        end angle=180,
        x radius=1.5cm,
        y radius =0.8cm
    ] ;
        \draw[thick] (5,0.3) arc
    [
        start angle=0,
        end angle=180,
        x radius=1.5cm,
        y radius =0.6cm
    ] ;
		\end{tikzpicture}};
		\node (8) at (8,-12) {
			\begin{tikzpicture}
				\foreach \i in {0,...,5}
	{
	\node (\i) at (\i,0) {\i };	
	}
    \draw[thick] (2,0.3) arc
    [
        start angle=0,
        end angle=180,
        x radius=1.0cm,
        y radius =0.5cm
    ] ;
    \draw[thick] (4,0.3) arc
    [
        start angle=0,
        end angle=180,
        x radius=0.5cm,
        y radius =0.4cm
    ] ;
        \draw[thick] (5,0.3) arc
    [
        start angle=0,
        end angle=180,
        x radius=2.0cm,
        y radius =0.7cm
    ] ;
		\end{tikzpicture}};
		\node (9) at (-8,-16) {
			\begin{tikzpicture}
				\foreach \i in {0,...,5}
	{
	\node (\i) at (\i,0) {\i };	
	}
    \draw[thick] (3,0.3) arc
    [
        start angle=0,
        end angle=180,
        x radius=1.0cm,
        y radius =0.5cm
    ] ;
    \draw[thick] (4,0.3) arc
    [
        start angle=0,
        end angle=180,
        x radius=2.0cm,
        y radius =0.8cm
    ] ;
        \draw[thick] (5,0.3) arc
    [
        start angle=0,
        end angle=180,
        x radius=1.5cm,
        y radius =0.7cm
    ] ;
		\end{tikzpicture}};
		\node (10) at (0,-16) {
			\begin{tikzpicture}
				\foreach \i in {0,...,5}
	{
	\node (\i) at (\i,0) {\i };	
	}
    \draw[thick] (2,0.3) arc
    [
        start angle=0,
        end angle=180,
        x radius=0.5cm,
        y radius =0.4cm
    ] ;
    \draw[thick] (4,0.3) arc
    [
        start angle=0,
        end angle=180,
        x radius=0.5cm,
        y radius =0.4cm
    ] ;
        \draw[thick] (5,0.3) arc
    [
        start angle=0,
        end angle=180,
        x radius=2.5cm,
        y radius =0.8cm
    ] ;
		\end{tikzpicture}};
		\node (11) at (8,-16) {
			\begin{tikzpicture}
				\foreach \i in {0,...,5}
	{
	\node (\i) at (\i,0) {\i };	
	}
    \draw[thick] (3,0.3) arc
    [
        start angle=0,
        end angle=180,
        x radius=1.5cm,
        y radius =0.7cm
    ] ;
    \draw[thick] (4,0.3) arc
    [
        start angle=0,
        end angle=180,
        x radius=1.0cm,
        y radius =0.5cm
    ] ;
        \draw[thick] (5,0.3) arc
    [
        start angle=0,
        end angle=180,
        x radius=2.0cm,
        y radius =0.8cm
    ] ;
		\end{tikzpicture}};
		\node (12) at (-4,-20) {\raisebox{0.6cm}{$\psi_1=$}
			\begin{tikzpicture}
				\foreach \i in {0,...,5}
	{
	\node (\i) at (\i,0) {\i };	
	}
    \draw[thick] (3,0.3) arc
    [
        start angle=0,
        end angle=180,
        x radius=1.0cm,
        y radius =0.5cm
    ] ;
    \draw[thick] (4,0.3) arc
    [
        start angle=0,
        end angle=180,
        x radius=1.0cm,
        y radius =0.5cm
    ] ;
        \draw[thick] (5,0.3) arc
    [
        start angle=0,
        end angle=180,
        x radius=2.5cm,
        y radius =0.8cm
    ] ;
		\end{tikzpicture}};
		\node (13) at (4,-20) {
			\begin{tikzpicture}
				\foreach \i in {0,...,5}
	{
	\node (\i) at (\i,0) {\i };	
	}
    \draw[thick] (3,0.3) arc
    [
        start angle=0,
        end angle=180,
        x radius=0.5cm,
        y radius =0.4cm
    ] ;
    \draw[thick] (4,0.3) arc
    [
        start angle=0,
        end angle=180,
        x radius=2.0cm,
        y radius =0.8cm
    ] ;
        \draw[thick] (5,0.3) arc
    [
        start angle=0,
        end angle=180,
        x radius=2.0cm,
        y radius =0.8cm
    ] ;
		\end{tikzpicture} \raisebox{0.6cm}{$=\psi_2$}};
	\draw[->,thick,darkblue] (N1) -- node[above left]{$T_1$} (N2) ;
	\draw[->,thick,darkblue] (N1) -- node[above right]{$T_3$} (N3) ;
	\draw[->,thick,darkblue] (N2) -- node[above left]{$\textcolor{purple}{T_0,} T_2$} (N4) ;
	\draw[->,thick,darkblue] (N2) -- node[below left]{$T_3$} (N5) ;
	\draw[->,thick,darkblue] (N3) -- node[below right]{$T_1$} (N5) ;
	\draw[->,thick,darkblue] (N3) -- node[above right]{$T_2, T_4$} (6) ;
	\draw[->,thick,darkblue] (N4) -- node[below left]{$T_3$} (7) ;
	\draw[->,thick,purple] (N5) -- node[above left]{$T_0$} (7) ;	
	\draw[->,thick,darkblue] (N5) -- node[above left]{$T_2$} (7andthreequarters) ;		
	\draw[->,thick,darkblue] (N5) -- node[above right]{$T_4$} (8) ;
	\draw[->,thick,darkblue] (6) -- node[below right]{$T_1$} (8) ;
	\draw[->,thick,darkblue] (7) -- node[above left]{$T_2$} (9) ;
	\draw[->,thick,darkblue] (7) -- node[below left, near end]{$T_4$} (10) ;
	\draw[->,thick,purple] (7andthreequarters) -- node[above left, near start]{$T_0$} (9) ;
	\draw[->,thick,darkblue] (7andthreequarters) -- node[above right, near start]{$T_4$} (11) ;
	\draw[->,thick,purple] (8) -- node[below right, near end]{$T_0$} (10) ;
	\draw[->,thick,darkblue] (8) -- node[above right]{$T_2$} (11) ;
	\draw[->,thick,darkblue] (9) -- node[below left]{$T_4$} (12) ;
	\draw[->,thick,darkblue] (9) -- node[below left, near start]{$T_1$} (13) ;
	\draw[->,thick,darkblue] (10) -- node[above left, near start]{$T_2$} (12) ;
	\draw[->,thick,purple] (11) -- node[below right]{$T_0$} (12) ;
	\draw[->,thick,darkblue] (11) -- node[below right]{$T_3$} (13) ;
\end{tikzpicture}}.
	\end{center}
	Here, we have drawn the arrows labeled only by $T_i$ with $i\leq 0$ in \textcolor{purple}{purple} to distinguish them from those that contribute in Theorem~\ref{thm:conj2}.
	
	First, observe that there are only two almost rainbow clans that we can reach, specifically the almost rainbow clans $\psi_1$ and $\psi_2$ at the bottom of the diagram above. Note that $\psi_1, \psi_2 \in \Psi_{2,2}$. Using only $T_i$ with $i>0$, there are exactly two paths from $\gamma_{u,u}$ to $\psi_1$. These paths are labeled by the sequences $(T_1, T_2, T_3, T_2, T_4)$ and $(T_1, T_2, T_3, T_4, T_2)$, both corresponding to the permutation $51324$. Thus, by Theorem~\ref{thm:conj2}, we have computed that $c_{2143,3142}^{51324} = 1$. 
	
	On the other hand, there are six paths from $\gamma_{u,u}$ to the almost rainbow clan $\psi_2$. These six paths are labeled by the sequences $\pi_1 = (T_1, T_3, T_2, T_4, T_3)$, $\pi_2 =(T_1, T_3, T_4, T_2, T_3)$, $\pi_3 = (T_3, T_1, T_2, T_4, T_3)$, $\pi_4 = (T_3, T_1, T_4, T_2, T_3)$, $\pi_5 = (T_3, T_2, T_1, T_2, T_3)$, and $\pi_6 = (T_3, T_4, T_1, T_2, T_3)$. The sequences $\pi_1,$ $\pi_2$, $\pi_3$, $\pi_4$, and $\pi_6$ all yield the permutation $41523$. Thus, by Theorem~\ref{thm:conj2}, we have computed that $c_{2143,3142}^{41523} = 1$. However, $\pi_5$ yields the permutation $4231$, so Theorem~\ref{thm:conj2} also gives $c_{2143,3142}^{4231} = 1$. Since these are the only paths from $\gamma_{u,u}$ to almost rainbow clans in $\Psi_{2,2}$, Theorem~\ref{thm:conj2} finally computes that $c_{2143,3142}^{w} = 0$ for all $w \notin \{51324, 41523, 4231 \}$.
	\end{example}

The fact that the products in Theorems~\ref{thm:conj1} and~\ref{thm:conj2} are multiplicity-free is remarkable. In contrast, for example, the Littlewood--Richardson rule shows that every nonnegative integer appears as a coefficient in some product of Grassmannian Schubert cycles. Indeed, the multiplicity-freeness of Theorems~\ref{thm:conj1} and~\ref{thm:conj2} does not extend to the product of subjacent Schubert cycles with each other, as illustrated by the following example.

\begin{example}
	\label{ex:prodsubjacent}
	Let $u=142536$ and $v=451236$.  Note that $u$ and $v$ are both $3$-inverse Grassmannian.  We have that $s_3u=132546$ and $s_3v=351246$ are subjacent.  Furthermore, we have
	\[\sigma_{132546}\cdot \sigma_{351246}=\sigma_{361425}+\sigma_{451326}+2\sigma_{461235},\]
	which is not multiplicity-free.
\end{example}

\section*{Acknowledgements}
We are grateful for conversations with Zach Hamaker and David E Speyer. Many thanks to Alexander Yong for bringing the reference \cite{Wyser} to our attention. We are also very grateful to Christian Gaetz, Patricia Klein, Allen Knutson, and Frank Sottile for many helpful comments on earlier drafts. In particular, we thank Frank Sottile for pointing out a derivation of Corollary~\ref{cor:monk} from Monk's formula and providing useful remarks on \cite{Bergeron.Sottile, Lenart.Robinson.Sottile}.  We also thank the anonymous referees for many helpful comments and suggestions.

O.P.\ was partially supported by a Mathematical Sciences Postdoctoral Research Fellowship (\#1703696) from the National Science Foundation, as well as by a Discovery Grant (RGPIN-2021-02391) and Launch Supplement (DGECR-2021-00010) from the Natural Sciences and Engineering Research Council of Canada. A.W.\ was partially supported by Bill Fulton's Oscar Zariski Distinguished Professor Chair funds and by National Science Foundation grant DMS-2344764.

\bibliographystyle{amsalphavar} 
\bibliography{WLR.bib}
\end{document}